\documentclass[11pt,letterpaper]{amsart}

\usepackage{amssymb, amsmath, amsfonts, latexsym}
\usepackage{enumerate}

\usepackage{color}
\usepackage{diagbox}
\newtheorem{thm}{Theorem}[]
\newtheorem{lem}{Lemma}[section]

\newtheorem{prop}{Proposition}[]

\newtheorem{rmk}{Remark}[section]
\theoremstyle{definition}

\usepackage{graphicx}
\usepackage{subcaption}
\numberwithin{equation}{section} \theoremstyle{remark}

\title[Universal edge asymptotics for Coulomb gases]{Precise universal edge asymptotics for planar $\beta=2$ Coulomb gases with radial external fields}

\author{Y\MakeLowercase{utao} M\MakeLowercase{a} and X\MakeLowercase{ujia} M\MakeLowercase{eng}}
\address{School of Mathematical Sciences $\&$ Laboratory  of Mathematics and Complex Systems of Ministry of Education, Beijing Normal University, 100875 Beijing, China.} 
\thanks{The research of Yutao Ma was supported in part by NSFC 12571149 and 985 Projects.}
\email{mayt@bnu.edu.cn, 202321130122@mail.bnu.edu.cn} 

\begin{document}
\maketitle

\begin{abstract}
We investigate the extremal statistics of planar $\beta=2$ Coulomb gases with radial external fields. For the rightmost eigenvalue and the spectral radius, we establish sharp Berry--Esseen bounds for their convergence to the Gumbel distribution, with explicit rates
\[
\frac{25\log\log n}{4e\log n}
\quad\text{and}\quad
\frac{2\log\log n}{e\log n},
\]
respectively. In addition, we derive sharp asymptotic equivalences for the large and moderate deviations of both statistics across all relevant scales. Analogous results hold for the smallest modulus. 
\end{abstract}



\section{Introduction}

The Coulomb gas model for $n$ points on the plane with external potential $Q: \mathbb{C}\to\mathbb{R}$ is the probability measure
\begin{equation}\label{1.1}
\frac{1}{Z_n}\prod_{j = 1}^{n}e^{-\frac{\beta n}{2}Q(z_{j})}\prod_{1\leq j< k\leq n}|z_{j} - z_{k}|^{\beta} d^2 z_j, \qquad z_1, \cdots, z_n\in\mathbb{C}, 
\end{equation}
where $Z_n$ is the normalization constant, $d^2 z$ is the two-dimensional Lebesgue measure, and $\beta>0$ is the inverse temperature. Under suitable assumptions on $Q$, the points accumulate with high probability on the support $S$ of an equilibrium measure $\mu$; see \cite{Hedenmalm2013}.

For limit theorems of the empirical measure of two-dimensional Coulomb gases, we refer to \cite{benarous1998, bordenave2012, chafai2013, hardy2012, guionnet2012} and the review \cite{Chafaii21}. 
In the special case $Q(z)=|z|^2$, which corresponds to the complex Ginibre ensemble, the edge behavior and extremal statistics have been extensively studied; see \cite{Cipolloni22Directional, HuMa2026, MaMeng26, Rider2003, Rider2004, Rider14}. 
Separately, Auster \cite{Auster} investigated the minimum and maximum modulus of the induced Ginibre ensembles. 
For general confining potentials $V$, the spectral radius was first treated in \cite{chafai14}, while hard-wall effects and intermediate deviation regimes were studied in \cite{Seo2022} and \cite{Lacroix2018}, respectively. 
Hole probabilities are addressed in \cite{AdhikariR, Adhikari18, Charlier26}. 
For the general potential-theoretic and random-matrix background, we refer to \cite{ABDF11, AGZ10, Dei99, DG09, For10, Hedenmalm2013, saff1997}. 

Of particular relevance to the present work is the seminal result of Chafa\"{\i} and P\'ech\'e \cite{chafai14}. 
They established that, for the $\beta=2$ Coulomb gas with a radially symmetric confining potential $V$ satisfying suitable assumptions, the spectral radius (maximum modulus) converges to the Gumbel distribution after appropriate centering and scaling. 
Their proof relies crucially on Kostlan's observation \cite{kostlan1992, Rider2003}, which decouples the radial parts into independent random variables and thereby reduces the extremal problem to a tractable one-dimensional analysis. More precisely, they showed that
$$
\max_{1\le i\le n}|\sigma_i| \stackrel{d}{=} \max_{1\le i\le n} Y_i,
$$
where $\{Y_i\}_{1\le i\le n}$ is an independent sequence with density proportional to $e^{-nV(r)} r^{2i-1}\mathbf{1}_{r>0}$. 
See also \cite{EbrahimiZohren}, who derived the Gumbel limit for both the spectral radius and the minimum modulus under slightly milder assumptions on $V,$  again via Kostlan's observation. However, Kostlan's observation approach fails for the rightmost eigenvalue (the maximum of the real parts).

Recently, \cite{MaMeng26} established the exact convergence rate of the spectral radius for the complex Ginibre ensemble, again leveraging Kostlan's observation. This was extended in \cite{HuMa2026} to both the spectral radius and the rightmost eigenvalue for real and complex Ginibre ensembles, using a first-order approximation of the integral operator restricted to suitable measurable subsets and exploiting the underlying determinantal (or Pfaffian) structure. Building on the universality result of \cite{Cipolloni23Universality}, the convergence rates were shown to be universal for complex random matrices with i.i.d.~entries under moment conditions. Parallel results for products of complex Ginibre ensembles and large chiral non-Hermitian matrices appear in \cite{MaMeng26a, MaMeng26b, MaWang26}. 

In this paper, we go beyond \cite{chafai14} by establishing sharp Berry--Esseen bounds for the Gumbel approximation (Theorem \ref{main}) and by deriving complete large and moderate deviation asymptotics for both the spectral radius and the rightmost eigenvalue (Theorems \ref{thmrealpart} and \ref{thmsp}) under modified assumptions on $V.$ We also establish similar results for the minimum modulus under the same conditions on $V$ in \cite{EbrahimiZohren}.
Our approach is based on the gap probability of the underlying determinantal point process, combined with a saddle-point analysis of the relevant Fredholm determinants, which allows us to extract precise asymptotic expansions across the full spectrum of deviation scales considered.

In the following, we always suppose that $\beta=2$ and $Q$ is radially symmetric, i.e.
\[
Q(z)=V(|z|),
\]
where $V:\mathbb{R}_{+}\to \mathbb{R}.$  
Then the joint density of $(\sigma_1,\dots,\sigma_n)$ is proportional to
\begin{equation}\label{jointpdf}
\prod_{i=1}^n e^{-nV(|z_i|)} \prod_{j<k} |z_j - z_k|^2.
\end{equation}
We assume that $V$ satisfies the following standard single-well conditions:
\begin{enumerate} 
\item[({\bf H}$_1$)] $V\in C^3(0, +\infty).$ 
\item[({\bf H}$_2$)] The equation $V'(t_u)=(2-u)/t_u$ has a unique solution $t_u\in(0,\infty)$ for each $u\in[0,2)$. 
\item[({\bf H}$_3$)] For each $u\in[0,2)$, the function $V(x)-(2-u)\log x$ has a unique global minimizer on $(0,\infty)$. 
\end{enumerate} 
Condition ({\bf H}$_2$) uniquely defines the inner and outer edges of the equilibrium measure's support, while ({\bf H}$_3$) ensures the stability and strict convexity of the associated variational problem. 
Let $t_0$ be the unique positive solution to $V'(t_0)=2/t_0$, and set
\[
\alpha := V''(t_0) + \frac{2}{t_0^2}.
\]
Under ({\bf H}$_1$)--({\bf H}$_3$), we have $\alpha>0$: indeed, defining $\phi(x)=xV'(x)$, we have
\[
(V(x)-2\log x)'=\frac{1}{x}(\phi(x)-2).
\]
The sign change of $(V(x)-2\log x)'$ at the global minimum forces $\phi$ to cross $2$ from below to above, hence $\phi'(t_0)>0$. Therefore,
\[
\alpha = V''(t_0)+\frac{2}{t_0^2}
= \frac{1}{t_0}\bigl(t_0V''(t_0)+V'(t_0)\bigr)
= \frac{\phi'(t_0)}{t_0} >0.
\]

\subsection{Results for the rightmost eigenvalue and the spectral radius}

Given
\[
\gamma_n=\frac{1}{2}\log\frac{n \alpha t_0^2}{8 \pi^4 (\log n)^5},
\qquad
\widetilde{\gamma}_n=\log \frac{n\alpha t_0^2}{8\pi (\log n)^2},
\]
we define the normalized statistics
\[
X_n:=\sqrt{\alpha\gamma_n n}(\max_{1\le i\le n} \Re \sigma_i-t_0-\frac{\sqrt{\gamma_n}}{\sqrt{\alpha n}}),
\qquad
\widetilde{X}_n:=\sqrt{\alpha\widetilde{\gamma}_n n}(\max_{1\le i\le n}|\sigma_i|-t_0-\frac{\sqrt{\widetilde{\gamma}_n}}{\sqrt{\alpha  n}}).
\]

We note that the centering and scaling for $\widetilde{X}_n$ in \cite{chafai14}, although also expressed in terms of \(t_0\) and \(\alpha\), are in a different form from ours; in contrast, the normalization in \cite{EbrahimiZohren} for $\widetilde{X}_n$ coincides with ours.

Our first main result gives the sharp convergence rates to the Gumbel law.

\begin{thm}\label{main} 	
Suppose $V$ satisfies $(${\bf H}$_1)$--$(${\bf H}$_3)$. Let $(\sigma_1,\dots,\sigma_n)$ have joint density \eqref{jointpdf}, and let $X_n, \widetilde{X}_n$ be as above. Then
\[
\sup_{t\in\mathbb{R}}|\mathbb{P}(X_n\le t)-e^{-e^{-t}}|
\sim \frac{25\log\log n}{4e\log n}\]
and
\[
\sup_{t\in\mathbb{R}}|\mathbb{P}(\widetilde{X}_n\le t)-e^{-e^{-t}}|
\sim \frac{2\log\log n}{e\log n}
\]
for $n$ large enough.  Here, $\sim$ means that the ratio of the left-hand side to the right-hand side tends to $1$ as $n \to \infty$.
\end{thm}

Next, we present the large and moderate deviation asymptotics for the rightmost eigenvalue.

\begin{thm}\label{thmrealpart}
Under the same assumptions as Theorem \ref{main}, we have the following asymptotic equivalences, classified according to the deviation scale.

(1) Large deviations (fixed displacement). For any fixed $t>t_0$,
\[
\mathbb{P}\bigl( \max_{1\le i\le n}\Re \sigma_i \ge t \bigr)
\sim
\frac{\sqrt{\alpha }\, t^2\, t_0 \exp(-nV(t)+nV(t_0))}{2\pi n(t^2-t_0^2) (tV'(t)-2)^{3/2}}
(\frac{t}{t_0})^{2n}.
\]

(2) Critical Gumbel scale and its refined correction.

(a) Pure Gumbel scale (fixed $t>1$):
\[
\mathbb{P}\bigl( \max_{1\le i\le n}\Re \sigma_i \ge t_0+\frac{t\sqrt{\gamma_n}}{\sqrt{\alpha n}} \bigr)
\sim
t^{-5/2}(\frac{8\pi^4 (\log n)^5}{\alpha n t_0^2 })^{\frac{t^2-1}{4}}.
\]

(b) Fine correction beyond the Gumbel scale ($1\ll v_n\ll \log n$):
\[
\mathbb{P}\bigl( \max_{1\le i\le n}\Re \sigma_i \ge t_0+\frac{\sqrt{\gamma_n}}{\sqrt{\alpha n}}+\frac{v_n}{\sqrt{\alpha n \gamma_n}} \bigr)
\sim
\exp\bigl(-v_n-\frac{v_n^2}{2\gamma_n}\bigr).
\]

(3) Moderate deviations on scale $d_n$. Given $d_n>0$:

\[\aligned
&\quad\mathbb{P}\bigl(\max_{1\le i\le n}\Re\sigma_i \ge t_0 + d_n \bigr)\\
&
\sim
\begin{cases}
\displaystyle \frac{\sqrt{t_0}}{4\pi \alpha n d_n^{5/2}}
\exp(-\frac{\alpha  n d_n^2}{2}), & \frac{\sqrt{\log n}}{\sqrt{n}} \ll d_n \ll n^{-\frac13},\\[1.2em]
\displaystyle \frac{\sqrt{t_0} \exp(-nV(t_0+d_n)+nV(t_0))}{4\pi \alpha n d_n^{5/2}}
(\frac{t_0+ d_n}{t_0})^{2n}, & n^{-\frac13}\lesssim d_n\ll 1.
\end{cases}
\endaligned\]
\end{thm}

The parallel result for the spectral radius is as follows.

\begin{thm}\label{thmsp}
Under the same assumptions as Theorem \ref{main}, we have the following asymptotic equivalences.

(1) Large deviations (fixed displacement). For any fixed $t>t_0$,
\[
\mathbb{P}\bigl( \max_{1\le i\le n}|\sigma_i| \ge t \bigr)
\sim
\frac{\sqrt{\alpha} \,t_0 t^2}{\sqrt{2\pi n}(t^2-t_0^2)(tV'(t)-2)}
\exp(-nV(t)+nV(t_0))(\frac{t}{t_0})^{2n}.
\]

(2) Critical Gumbel scale and its refined correction.

(a) Pure Gumbel scale (fixed $t>1$):
\[
\mathbb{P}\bigl( \max_{1\le i\le n}|\sigma_i| \ge t_0+\frac{t\sqrt{\widetilde{\gamma}_n}}{\sqrt{\alpha n}} \bigr)
\sim
t^{-2}(\frac{8\pi (\log n)^2}{\alpha n t_0^2 })^{\frac{t^2-1}{2}}.
\]

(b) Fine correction beyond the Gumbel scale ($1\ll v_n\ll \log n$):
\[
\mathbb{P}\bigl( \max_{1\le i\le n}|\sigma_i| \ge t_0+\frac{\sqrt{\widetilde{\gamma}_n}}{\sqrt{\alpha n}}+\frac{v_n}{\sqrt{\alpha n \widetilde{\gamma}_n}} \bigr)
\sim
\exp\bigl(-v_n-\frac{v_n^2}{2\widetilde{\gamma}_n}\bigr).
\]

(3) Moderate deviations on scale $d_n$. Given $d_n>0$:
\[\aligned
&\quad\mathbb{P}\bigl(\max_{1\le i\le n}|\sigma_i| \ge t_0 + d_n \bigr)\\
&
\sim
\begin{cases}
\displaystyle \frac{t_0}{2\sqrt{2\pi \alpha n} \, d_n^2 }
\exp\bigl(-\frac{1}{2}\alpha n d_n^2\bigr), & \frac{\sqrt{\log n}}{\sqrt{n}} \ll d_n \ll n^{-\frac13},\\[1.2em]
\displaystyle \frac{t_0\exp(-n V(t_0+d_n)+nV(t_0))}{2\sqrt{2\pi \alpha n} \,  d_n^2 }
(\frac{t_0+d_n}{t_0})^{2n}, & n^{-\frac13}\lesssim d_n\ll 1.
\end{cases}
\endaligned\]
\end{thm}

\begin{rmk}
The choice $V(x)=x^2$ on $x\ge0$ gives $t_0=1$ and $\alpha=4$, corresponding to the complex Ginibre ensemble. In this setting, the convergence rates for $\widetilde{X}_n$ and $X_n$ match those obtained in \cite{HuMa2026, MaMeng26}. For the large and moderate deviation estimates, \cite{XuZeng2026} derived lower and upper bounds differing by a factor $1/n$; our result shows that their upper bound is sharp, yielding asymptotic equivalence.
\end{rmk}

\begin{rmk}\label{induced}
For $V(x)=x^2-2\kappa\log x$ on $x>0$ with $\kappa>0$, we have $t_0=\sqrt{1+\kappa}$ and $\alpha=4$. Theorem \ref{main} gives, for any fixed $t\in\mathbb{R}$,
\[
\lim_{n\to\infty}\mathbb{P}(\max_{1\le i\le n}|\sigma_i|\le \sqrt{1+\kappa}+\frac{\sqrt{\gamma_n'}}{\sqrt{4n}}+\frac{t}{\sqrt{4\gamma_n' n}})=e^{-e^{-t}},
\]
where $\gamma_n'=\log \frac{n(1+\kappa)}{2\pi (\log n)^2}$, recovering \cite[Theorem 6]{Auster} for the complex induced Ginibre ensemble with $L=\kappa n$. Beyond this known limit, Theorem \ref{main} further provides the analogous statement for the real parts,
\[
\lim_{n\to\infty}\mathbb{P}(\max_{1\le i\le n}\Re\sigma_i\le \sqrt{1+\kappa}+\frac{\sqrt{\gamma_n}}{\sqrt{4n}}+\frac{t}{\sqrt{4\gamma_n n}})=e^{-e^{-t}},
\]
with $\gamma_n=\frac{1}{2}\log\frac{n(1+\kappa)}{2\pi^4 (\log n)^5}$; Theorems \ref{thmrealpart} and \ref{thmsp} provide the corresponding large and moderate deviation principles. \end{rmk}

\begin{rmk}\label{Vn}
We make some remarks concerning the assumptions on \(V\).

\begin{enumerate} 
\item Our assumptions remove the need for strict convexity of \(V\), thereby accommodating, for example, the power-law case \(V(r)=r^a\) with \(0<a<1\), the double-well case \(V(r)=r^4-\kappa r^2\) (\(\kappa\in\mathbb{R}\)), and the \(n\)-dependent case \(V_n(r)=(1+\frac{\kappa}{n})\log(1+r^2)\).

\item In \cite{chafai14}, the authors assume \(V\) to be smooth and strongly convex (\(V''>a>0\)), together with condition (H2), which excludes the power-law case with \(0<a<1\). These additional regularity requirements are stronger than our assumptions $(${\bf H}$_1)$--$(${\bf H}$_3)$.

\item In \cite{EbrahimiZohren}, the authors assume either the monotonicity of \(rV'(r)\) on the entire half-line \((0,\infty)\), or the stronger hypothesis that \(V'(r)>0\) and \(V\) is convex, both of which exclude the double-well case \(V(r)=r^4-\kappa r^2\) (\(\kappa\in\mathbb{R}\)). By contrast, our conditions $(${\bf H}$_2)$--$(${\bf H}$_3)$ are equivalent to the strict increase of \(rV'(r)\) on the restricted set \(\{r: 0<rV'(r)\le 2\}\). This is precisely the minimal natural assumption needed to produce the edge Gumbel behavior when \(\alpha>0\).
\item The theorems are stated for a fixed potential $V$ independent of $n$. For $n$-dependent potentials $V=V_n$, the method remains applicable; however, one must carefully track the limiting behavior of $t_0$, $\alpha$, and the leading-order terms of $\gamma_n$ and $\gamma_n'$. 
\end{enumerate}
\end{rmk}

\begin{rmk}\label{listV}
We list some classical potentials satisfying $(${\bf H}$_1)$--$(${\bf H}$_3)$:
\begin{enumerate}
\item $V(r)=r^{a}$, $a>0$: $t_0=(2/a)^{1/a}$, $\alpha=2^{1-2/a}a^{1+2/a}$.
\item $V(r)=r^4-\kappa r^2$, $\kappa\in\mathbb{R}$: 
\(
t_0=\frac12\sqrt{\kappa+\sqrt{\kappa^2+8}}, \; 
\alpha=4\sqrt{\kappa^2+8}.
\)
\item $V(r)=\frac12 r^4-2\kappa\log r, \kappa>0$: 
$t_0=(1+\kappa)^{1/4},\; \alpha=8\sqrt{1+\kappa}$.
\item $V(r)=r^2+r^{-2}$: 
$t_0=\sqrt{\frac{1+\sqrt5}{2}},\; \alpha=10-2\sqrt5$.
\item $V(r)=r^2+\log(1+r^2)$: 
$t_0=\sqrt{\frac{\sqrt5-1}{2}},\quad \alpha=10-2\sqrt5$.
\end{enumerate}
\end{rmk}

\subsection{Results for $\min_{1\le i\le n}\Re \sigma_i$ and $\min_{1\le i\le n}|\sigma_i|$ }

We now assume that $V\in C^3(0,+\infty)$ satisfies the following conditions used in \cite{EbrahimiZohren}: $V'(r)=0$ has a unique solution $t_1>0$, and $rV'(r)$ is strictly increasing on $(0,+\infty)$.

Since $r V'(r)$ is strictly increasing, we know 
$$0<(r V'(r))'|_{r=t_1}=t_1V''(t_1)+V'(t_1)=t_1 V''(t_1)$$ and then 
 $V''(t_1)>0.$ Given
\[
{\tau}_n=\log \frac{nV''(t_1) t_1^2}{8\pi (\log n)^2},
\]
it was showed in \cite{EbrahimiZohren} that 
$$\lim_{n\to\infty}\mathbb{P}(\min_{1\le i\le n}|\sigma_i|\ge t_1-\frac{\sqrt{\tau_n}}{\sqrt{V''(t_1) n}}+\frac{t}{\sqrt{V''(t_1)n\tau_n}})=e^{-e^{t}}$$ for any $t\in\mathbb{R}.$

Now we state similar results for $\min_{1\le i\le n}|\sigma_i|.$
\begin{thm}\label{minBerry}
Let $V$ be in  $C^3(0, +\infty)$ and suppose that $V'=0$ has a unique solution $t_1>0$ and $r V'(r)$ is strictly increasing on $(0, +\infty).$	Let $t_1, \tau_n$ be defined as above. Then 
\begin{itemize}
\item[(1)] Berry-Esseen bound:
$$\aligned 
\sup_{t\in\mathbb{R}}|\mathbb{P}(\min_{1\le i\le n}|\sigma_i|\ge t_1-\frac{\sqrt{\tau_n}}{\sqrt{V''(t_1) n}}+\frac{t}{\sqrt{V''(t_1)n\tau_n}})-e^{-e^{t}}|&\sim\frac{2\log\log n}{e\log n}.
\endaligned $$ 
\item[(2)] For any $0<t<t_1,$ we have 
$$\mathbb{P}(\min_{1\le i\le n}|\sigma_i|\le t)\sim \frac{\sqrt{V''(t_1)}\,t_1 t}{\sqrt{2\pi n}(t_1^2-t^2)(-V'(t))}
\exp(-nV(t)+nV(t_1)).$$ 
\item[(3)] Given $t>1,$ we obtain \[
\mathbb{P}\bigl( \min_{1\le i\le n}|\sigma_i| \le t_1-\frac{t\sqrt{{\tau}_n}}{\sqrt{V''(t_1) n}} \bigr)
\sim
t^{-2}(\frac{8\pi (\log n)^2}{V''(t_1) n t_1^2 })^{\frac{t^2-1}{2}}.
\]
\item[(4)] Given $1\ll v_n\ll \log n$:
\[
\mathbb{P}\bigl( \min_{1\le i\le n}|\sigma_i| \le t_1-\frac{\sqrt{{\tau}_n}}{\sqrt{V''(t_1) n}}-\frac{v_n}{\sqrt{V''(t_1) n {\tau}_n}} \bigr)
\sim
\exp\bigl(-v_n-\frac{v_n^2}{2{\tau}_n}\bigr).
\]
\item[(5)] Given $d_n>0$:
\[\aligned
&\quad\mathbb{P}\bigl(\min_{1\le i\le n}|\sigma_i| \le t_1 - d_n \bigr)\\
&
\sim
\begin{cases}
\displaystyle \frac{t_1}{2\sqrt{2\pi V''(t_1) n} \, d_n^2 }
\exp\bigl(-\frac{1}{2}V''(t_1) n d_n^2\bigr), & \frac{\sqrt{\log n}}{\sqrt{n}} \ll d_n \ll n^{-\frac13},\\[1.2em]
\displaystyle \frac{t_1\exp(-n V(t_1-d_n)+nV(t_1))}{2\sqrt{2\pi V''(t_1) n} \,  d_n^2 }
, & n^{-\frac13}\lesssim d_n\ll 1.
\end{cases}
\endaligned\]
\end{itemize} 
\end{thm} 

\begin{rmk} The assumption that the solution to $V'(r)=0$ is positive in Theorem \ref{minBerry} exclude all $V$ listed above with $V'(0)=0.$ For other examples satisfy the conditions $V'(t_1)=0$ with $t_1>0$ and $rV'(r)$ strictly increasing on $(0, +\infty)$, we have   
\begin{enumerate}
\item $V(r)=\frac12 r^4-2\kappa\log r, \kappa>0$: 
$t_1=\kappa^{1/4},\; V''(t_1)=8\sqrt{\kappa}$.
\item $V(r)=r^2+r^{-2}$: 
$t_1=1,\; V''(t_1)=8$.
\item $V(r)=r^2-2\kappa \log r, t_1=\sqrt{\kappa}$ and $V''(t_1)=4.$
\end{enumerate} 	
\end{rmk}

\begin{rmk} 
The radiality of $V$ implies that the correlation kernel $K_n$ defined in \eqref{Kn} below satisfies the symmetry
\[
K_n(z,w)=K_n(-z,-w),
\]
since both $|z|$ and $z\bar w$ are invariant under $z\mapsto -z$. Consequently, the change of variables $z\mapsto -z$ maps the half-plane $\{\Re z\ge x\}$ onto $\{\Re z\le -x\}$ and induces a unitary equivalence between the restricted operators $\mathbb K_n|_{\{\Re z\ge x\}}$ and $\mathbb K_n|_{\{\Re z\le -x\}}$. Hence their Fredholm determinants coincide:
\[
\det\bigl({\rm I}-\mathbb K_n|_{\{\Re z\ge x\}}\bigr)
=
\det\bigl({\rm I}-\mathbb K_n|_{\{\Re z\le -x\}}\bigr).
\]
By the hole probability formula, this yields the distributional identity
\[
\min_{1\le i\le n}\Re \sigma_i \stackrel{d}{=} -\max_{1\le i\le n}\Re \sigma_i.
\]
Therefore, under the same conditions, the results for the rightmost eigenvalue in Theorems \ref{main} and \ref{thmrealpart} apply equally to $\min_{1\le i\le n}\Re \sigma_i$.
\end{rmk}

We now sketch the proof strategy.
\subsection{Sketch of the proof}
First, we verify that $(\sigma_1,\dots,\sigma_n)$ forms a determinantal point process with correlation kernel
\begin{equation}\label{Kn}
K_n(z,w)=
\exp(-\frac{n}{2}V(|z|)-\frac{n}{2}V(|w|))
\sum_{k=1}^{n} \frac{(z\bar w)^{k}}{J_k}
\end{equation}
with $$J_k:=\int_{\mathbb{C}} e^{-nV(|z|)} |z|^{2k}\,dz.$$
Let $\mathbb{K}_n$ denote the integral operator on $L^2(\mathbb{C})$ with kernel $K_n$. For any measurable $A\subset\mathbb{C}$, the restricted operator $\mathbb{K}_n|_A$ is trace class and Hilbert--Schmidt, and the gap probability formula gives
\[
\mathbb{P}(\{\sigma_1,\dots,\sigma_n\}\cap A=\varnothing)
= \det({\rm I}-\mathbb{K}_n|_A).
\]
Moreover, from \cite{Gohberg} we have the bound
\begin{equation}\label{comp}
\left|\det({\rm I}-\mathbb{K}_n|_A) - \exp\bigl(-\operatorname{Tr}(\mathbb{K}_n|_A)\bigr)\right|
\le \|\mathbb{K}_n|_A\|_2\,
\exp\{\frac12(\|\mathbb{K}_n|_A\|_2+1)^2 - \operatorname{Tr}(\mathbb{K}_n|_A)\}.
\end{equation}
Thus, whenever $\|\mathbb{K}_n|_A\|_2\ll 1$, we have
\[
\det({\rm I}-\mathbb{K}_n|_A) \sim \exp\bigl(-\operatorname{Tr}(\mathbb{K}_n|_A)\bigr).
\]

For Theorem \ref{main}, we choose suitable sets $A(t)$ and $\widetilde{A}(t)$ such that
\[
\mathbb{P}(X_n\le t)=\det({\rm I}-\mathbb{K}_n|_{A(t)}),\qquad
\mathbb{P}(\widetilde{X}_n\le t)=\det({\rm I}-\mathbb{K}_n|_{\widetilde{A}(t)}).
\]
The estimates of the trace and Hilbert--Schmidt norms rely on Laplace's method and saddle-point analysis, justified by assumptions (H1)--(H3). Sharp control of these norms yields the Gumbel limit with the exact convergence rate.

For Theorem \ref{thmrealpart}, the key observation is the hole-probability identity
\[
\mathbb{P}\bigl(\max_{1\le i\le n}\Re \sigma_i \ge t_0 + z_n\bigr)
= 1 - \det\bigl({\rm I}- \mathbb{K}_n|_{\{\Re z > t_0 + z_n\}}\bigr),
\]
where $z_n\ll 1$ in the moderate regime and $z_n=t-t_0$ in the large-deviation regime. 
Once we establish
\[
\operatorname{Tr}\bigl(\mathbb{K}_n|_{\{\Re z > t_0 + z_n\}}\bigr)
= \int_{\Re z > t_0 + z_n} K_n(z,z)\,dz = o(1),
\]
the inequality $|K_n(z,w)|^2 \le K_n(z,z)K_n(w,w)$ implies
\[
\|\mathbb{K}_n|_{\{\Re z > t_0 + z_n\}}\|_2^2
\le \bigl(\operatorname{Tr}(\mathbb{K}_n|_{\{\Re z > t_0 + z_n\}})\bigr)^2 = o(1).
\]
Consequently, \eqref{comp} yields the crucial asymptotic equivalence
\begin{equation}\label{keycom}
\begin{aligned}
1 - \det\bigl({\rm I} - \mathbb{K}_n|_{\{\Re z > t_0 + z_n\}}\bigr)
&\sim 1 - \exp\bigl(-\operatorname{Tr}\bigl(\mathbb{K}_n|_{\{\Re z > t_0 + z_n\}}\bigr)\bigr) \\
&\sim \operatorname{Tr}\bigl(\mathbb{K}_n|_{\{\Re z > t_0 + z_n\}}\bigr).
\end{aligned}
\end{equation}
This recovers the two-sided inequality used in \cite{MaMeng26b, XuZeng2026}:
\[
n^{-1}\operatorname{Tr}\bigl(\mathbb{K}_n|_{\{\Re z > t_0 + z_n\}}\bigr)
\le
\mathbb{P}\bigl(\max_{1\le i\le n}\Re \sigma_i \ge t_0 + z_n\bigr)
\le
\operatorname{Tr}\bigl(\mathbb{K}_n|_{\{\Re z > t_0 + z_n\}}\bigr).
\]
Our new asymptotic formula \eqref{keycom} improves upon these bounds by eliminating the extraneous $n^{-1}$ factor, thus yielding the correct leading-order tail behavior. Similar argument works for Theorem \ref{thmsp}. 

The same argument applies to the smallest modulus. In fact, the behavior near the inner edge $t_1$ is completely analogous to that near the outer edge $t_0$, with $t_0$ replaced by $t_1$ and $\alpha:=V''(t_0)+2/t_0^2$ replaced by $V''(t_1)$. Thus the convergence rates coincide, and the deviation formulas follow from the same saddle-point analysis mutatis mutandis. The only technical difference is that the dominant trace contribution comes from $n-\sqrt{n\log n}\le k\le n$ for the maximum, whereas for the minimum it comes from $0\le k\le \sqrt{n\log n}$; apart from this, the two analyses are entirely parallel.

Throughout the paper, we use the standard asymptotic notation. Given positive sequence $(b_n)$ and real sequence $(a_n)$, we write
\[
a_n = O(b_n) \quad\text{if } \exists\, C>0 \text{ such that } |a_n| \le C b_n \text{ for all sufficiently large } n,
\]
and
\[
a_n = o(b_n) \quad\text{if } \lim_{n\to\infty} \frac{a_n}{b_n} = 0.
\]
We write \(a_n \ll b_n\) to mean \(a_n = o(b_n)\),  \(a_n \lesssim b_n\) to mean \(a_n = O(b_n)\) and $a_n\asymp b_n$ to mean $a_n\lesssim b_n$ as well as $b_n\lesssim a_n$. The notation \(a_n \sim b_n\) signifies \(a_n/b_n \to 1\) as \(n\to\infty\). All asymptotic relations are understood as \(n\to\infty\), unless otherwise specified. 
In this paper, all symbols carrying a tilde (e.g., $\widetilde{A}(t)$, $\widetilde{X}_n$) refer to quantities associated with the spectral radius, whereas symbols without a tilde (e.g., $A(t)$, $X_n$) generally correspond to the rightmost eigenvalue.

The paper will be organized as follows: in the next section, we provide some lemmas and the third section is devoted to the proof of Theorems \ref{main}-\ref{thmsp} and we put the proof of Theorem \ref{minBerry} at the last section.  

\section{Preliminaries}  
It is known that the tuple $(\sigma_1, \dots, \sigma_n)$ forms a determinantal point process (see, e.g., \cite{Cronvall, hough2009}). Since this fact serves as the starting point for our analysis, we present it as a separate lemma for clarity, although it follows directly from the aforementioned references.\begin{lem}\label{lemdpp}
	Let 
\begin{equation}\label{Cnk}
J_{k} = \int_{\mathbb{C}} e^{-nV(|z|)} |z|^{2k} \, dz.
\end{equation}
Then $(\sigma_1, \dots, \sigma_n)$ form a determinantal point process with correlation kernel given by
\[
K_n(z, w) = \exp\!\left( -\frac{n}{2}V(|z|) - \frac{n}{2}V(|w|) \right) \sum_{k=0}^{n-1} \frac{(z \bar{w})^{k}}{J_{k}}.
\]
\end{lem}

Let $\mathbb{K}_n$ denote the integral operator on $L^2(\mathbb{C})$ with kernel $K_n$, defined by $$(\mathbb{K}_n f)(z)=\int_{\mathbb{C}} K_n(z, w) f(w) \, d w.$$ Invoking Lemma \ref{lemdpp} and the properties of the gap probability, we express the probability of interest as the Fredholm determinant of $\mathbb{K}_n$ restricted to a measurable set. Setting  
\[
\widetilde{A}(t) = \{ z \in \mathbb{C} : |z| > \widetilde{a}(t)\} \quad \text{and} \quad {A}(t) = \{ z \in \mathbb{C} : \Re z > a(t)\}
\]
with $\widetilde{a}(t)=t_0 + \frac{\sqrt{\widetilde{\gamma}_n}}{\sqrt{\alpha n}}+\frac{t }{\sqrt{\alpha n\widetilde{\gamma}_n}}$ and ${a}(t)=t_0 + \frac{\sqrt{{\gamma}_n}}{\sqrt{\alpha n}}+\frac{t }{\sqrt{\alpha n{\gamma}_n}},$ 
we know from the gap probability formula that 
$$\mathbb{P}(X_n\le t)=\det({\rm I}-\mathbb{K}_n|_{A(t)}) \quad\text{and}\quad \mathbb{P}(\widetilde{X}_n\le t)=\det({\rm I}-\mathbb{K}_n|_{\widetilde{A}(t)}).$$ 

As discussed in the introduction, we need to analyze the trace and the Hilbert-Schmidt norm of the integral operator $\mathbb{K}_n$ restricted to $A(t)$ and $\widetilde{A}(t)$. To this end, we must derive asymptotics for certain involved quantities. We begin by presenting a lemma concerning $V$, which will be frequently utilized in our analysis.  

\begin{lem}\label{lemphi} Suppose that $V$ satisfies the conditions $(${\bf H}$_1)$--$(${\bf H}$_3).$ 	Given $z_n=o(1)$ and $0\le u\ll 1.$ Define $$\varphi_u(r) = V(r) - (2-u)\ln r.$$  
Then, $\varphi_u$ has a unique solution to the equation $\varphi_u'=0,$ which is denoted by $t_u.$  We have the following asymptotics \\
$(1):$ $\varphi_u'(t_0+z_n)=\frac{u}{t_0}+\alpha z_n+O(z_n^2+|u z_n|)$ and $$\varphi_u(t_0+z_n)=\varphi_0(t_0)+\frac12\alpha z_n^2+u\log(t_0+z_n)+O(|z_n|^3+|u|z_n^2).$$ 
$(2):$ $\varphi''_u(t_u)=\alpha+O(|u|)$ and $$\varphi_u(t_u)=\varphi_u(t_0)-\frac{u^2}{2\alpha t_0^2}+O(|u|^3).$$
\end{lem}
\begin{proof} $\varphi_u'(r)=V'(r)-\frac{2-u}{r}.$ The condition $(${\bf H}$_2)$ yields that the equation $\varphi_u'=0$ has a unique solution. Recall $V'(t_0)=\frac{2}{t_0}$ and $\alpha=V''(t_0)+\frac{2}{t_0^2}.$ We know
\begin{equation}\label{varphi12}\aligned\varphi_u'(t_0)&=V'(t_0)-\frac{2-u}{t_0}=\frac{u}{t_0}; \\
\varphi''_u(t_0)&=V''(t_0)+\frac{2-u}{t_0^2}=\alpha-\frac{u}{t_0^2}. \endaligned\end{equation} 
Since $z_n=o(1),$ we apply the Taylor formula to have 
$$\aligned \varphi_0(t_0+z_n)
&=\varphi_0(t_0)+\frac{\alpha z_n^2}{2}+O(|z_n|^3).\endaligned $$
Thus, it follows from definition that $$\aligned \varphi_u(t_0+z_n)&=\varphi_0(t_0+z_n)+u \log (t_0+z_n)\\
&=\varphi_0(t_0)+\frac{\alpha z_n^2}{2}+u\log (t_0+z_n)+O(|z_n^3|). \endaligned $$
Similarly,
$$\aligned \varphi_u'(t_0+z_n)&=\varphi_u'(t_0)+\varphi_u''(t_0)z_n+O(z_n^2)\\
&=\frac{u}{t_0}+\alpha z_n+O(z_n^2+|u z_n|). 
\endaligned $$ 
Since \(u=o(1)\) and \(2-u=V'(t_u)t_u\), the continuity of the involved functions implies that \(t_u=t_0+o(1)\). Applying Taylor's formula, we obtain
\[
\begin{aligned}
2-u
&=t_0V'(t_0)+\bigl(t_0V''(t_0)+V'(t_0)\bigr)(t_u-t_0)+O\bigl(|t_u-t_0|^2\bigr)\\
&=2+t_0\alpha\,(t_u-t_0)+O\bigl(|t_u-t_0|^2\bigr),
\end{aligned}
\]
which yields \(t_u=t_0+O(|u|)\). Consequently, we have the asymptotic expansion
\begin{equation} \label{txasym}
t_u=t_0-\frac{u}{\alpha t_0}+O(u^2).
\end{equation}
Thereby, we derive from \eqref{varphi12} and \eqref{txasym} that  $$\aligned \varphi_u''(t_u)&=\varphi_u''(t_0)+O(|t_u-t_0|)=\alpha+O(|u|)\endaligned $$
and 
$$\aligned\varphi_u(t_u)
&=\varphi_u(t_0)+\frac{u}{t_0}(t_u-t_0)+\frac{1}{2}(\alpha-\frac{u}{ t_0^2})(t_u-t_0)^2+O(|u|^3)\\
&=\varphi_u(t_0)-\frac{u^2}{2\alpha t_0^2}+O(|u|^3). \endaligned $$ 
The proof is completed now. 
\end{proof}

Next, we derive asymptotic expansions for several important quantities using Lemma~\ref{lemphi} together with saddle-point analysis and Laplace's method.
\begin{lem}\label{lemjk}
Let $J_k$ be defined as above.  
Then, as $n \to \infty$,  
\[
J_{n-k}=2\pi \sqrt{\frac{2\pi}{n\alpha}}\exp(-nV(t_0)+\frac{2k^2}{\alpha n t_0^2}) t_0^{2n-2k+1}(1+O(n^{-1/2}\sqrt{\log n}))
\] 
for $1\le k\lesssim \sqrt{n\log n}.$ 
\end{lem}

\begin{proof}
Write $z = r e^{i\theta}$ with $r \ge 0$, $\theta \in [0,2\pi)$. Then $dz = r\,dr\,d\theta$ and  
\[
J_{n-k}^{} = \int_0^{2\pi} d\theta \int_0^\infty \exp\!\big( -n V(r) + (2n-2k+1)\ln r \big) \, dr.
\]  
The angular integral gives a factor $2\pi$. Thus  
\begin{equation}\label{Jk}
J_{n-k}^{} = 2\pi \int_0^\infty \exp\!\big( -n V(r) + (2n-2k+1)\ln r \big) \, dr=2\pi \int_0^\infty \exp\!\big( -n \varphi_u(r)\big) \, dr
\end{equation}
with $u=\frac{2k-1}{n}.$ Lemma \ref{lemphi}, the Laplace method and the Guassian integral yield 
$$\aligned J_{n-k}&=2\pi (1+O(n^{-1}))e^{-n\varphi_u(t_u)}\int_{-\infty}^\infty e^{-\frac{n}{2}\varphi_u''(t_u)(r-t_u)^2} dr \\
&=2\pi (1+O(n^{-1}))e^{-n\varphi_u(t_u)}\sqrt{\frac{2\pi}{n\varphi''_u(t_u)}}\\
&=2\pi e^{-n V(t_0)}t_0^{2n-2k+1}\sqrt{\frac{2\pi}{n\alpha}}\exp(\frac{2k^2}{\alpha n t_0^2})(1+o((\log n)^{-1})).\endaligned $$ 
Here, for the last equality we use the expression 
$$\frac{n u^2}{2\alpha t_0^2}=\frac{(2k-1)^2}{2\alpha n t_0^2}=\frac{2k^2}{\alpha n t_0^2}+O(n^{-1/2}\sqrt{\log n}).$$
This completes the proof.
\end{proof}

\begin{lem}\label{lemjktilde} 
For  \(t\in\mathbb{R}\) fixed, set  $\widetilde{a}(t)=t_0 + \frac{\sqrt{\widetilde{\gamma}_n}}{\sqrt{\alpha n}}+\frac{t }{\sqrt{\alpha n\widetilde{\gamma}_n}}$ and 
\[
\widetilde{A}(t) = \{ z \in \mathbb{C} : |z| > \widetilde{a}(t)\}.
\]  
Define \[
\widetilde{J}_{k}(t) = \int_{\widetilde{A}(t)} \exp(-n V(|z|))  |z|^{2k} \, dz
\]  
for $0\le k\le n-1.$
Then, as \(n \to \infty\),  
\[
\widetilde{J}_{n-k}(t)\sim\frac{
  2\pi  t_0^{2n}e^{-nV(t_0)}
  \bigl(1+o((\log n)^{-1})\bigr)
}{
  \sqrt{\alpha n \widetilde{\gamma}_n}(\widetilde{a}(t))^{2k-1}
  (
    1+\frac{t}{\widetilde{\gamma}_n}
    +\frac{2k}{t_0\sqrt{\alpha n \widetilde{\gamma}_n}})
}
\exp(-\frac{\widetilde{\gamma}_n}{2}
  (1+\frac{t}{\widetilde{\gamma}_n})^2
  ).
\]
uniformly on $|t|\lesssim (\log n)^{1/4}$ and $1\le k\lesssim \sqrt{n\log n}.$ 
\end{lem} 

\begin{proof}
Similarly as for $J_{n-k},$ using the polar coordinates, we have 
\begin{equation}\label{Jktilde}\aligned
\widetilde{J}_{n-k}(t)&=\int_0^{2\pi} d\theta \int_{\widetilde{a}(t)}^\infty \exp\!\big( -n V(r) + (2n-2k+1)\ln r \big)  \, dr\\
&=2\pi  \int_{\widetilde{a}(t)}^\infty \exp\!\big( -n \varphi_u(r) \big)  \, dr 
\endaligned 
\end{equation}
with $u=\frac{2k-1}{n}.$
The expression \eqref{txasym} implies that $t_u<t_0$ and then $t_u<\widetilde{a}(t).$ Thus, the function $\varphi_u$ attains its minimum at $\widetilde{a}(t)$ and the Laplace method yields  
$$\widetilde{J}_{n-k}(t)=2\pi (1+O(n^{-1}))e^{-n\varphi_u(\widetilde{a}(t))}\frac{1}{n\varphi'_u(\widetilde{a}(t))}.$$ Taking $z_n=\frac{\sqrt{\widetilde{\gamma}_n}}{\sqrt{\alpha n}}(1+\frac{t }{\widetilde{\gamma}_n}),$ which is of order $n^{-1/2}\sqrt{\log n},$ we have from Lemma \ref{lemphi} that  $$\aligned n\varphi'_u(\widetilde{a}(t))&=\frac{n u}{t_0}+\sqrt{\alpha n\widetilde{\gamma}_n}(1+\frac{t}{\widetilde{\gamma}_n}) +O(\log n) \\
&=\sqrt{\alpha n\widetilde{\gamma}_n}(1+\frac{t}{\widetilde{\gamma}_n}+\frac{2k}{t_0\sqrt{\alpha n \widetilde{\gamma}_n}})(1+o((\log n)^{-1}))\endaligned $$ and 
$$\aligned n\varphi_u(\widetilde{a}(t)))&=n\varphi_0(t_0)+\frac{\widetilde{\gamma}_n}{2} (1+\frac{t }{\widetilde{\gamma}_n})^2+(2k-1)\log \widetilde{a}(t)+O(n^{-1/2}(\log n)^{3/2}). \endaligned$$
Thus, 
$$e^{-n\varphi_u(\widetilde{a}(t))}
=
(\widetilde{a}(t))^{-(2k-1)} t_0^{2n}
\exp(-nV(t_0)-\frac{\widetilde{\gamma}_n}{2}
(1+\frac{t}{\widetilde{\gamma}_n})^2)
(1+o((\log n)^{-1})).$$ This gives 
$$\aligned \widetilde{J}_{n-k}(t)\sim 
\frac{
  2\pi  t_0^{2n}e^{-nV(t_0)}
  \bigl(1+o((\log n)^{-1})\bigr)
}{
  \sqrt{\alpha n \widetilde{\gamma}_n}(\widetilde{a}(t))^{2k-1}
  (
    1+\frac{t}{\widetilde{\gamma}_n}
    +\frac{2k}{t_0\sqrt{\alpha n \widetilde{\gamma}_n}})
}
\exp(-\frac{\widetilde{\gamma}_n}{2}
  (1+\frac{t}{\widetilde{\gamma}_n})^2
  )).\endaligned $$
\end{proof} 

\begin{lem}\label{lemjkt} 
Set  $a(t)=t_0 + \frac{\sqrt{\gamma_n}}{\sqrt{\alpha n}}+\frac{t }{\sqrt{\alpha n\gamma_n}}$ and 
\[
A(t) = \{ z \in \mathbb{C} : \Re z > a(t)\}.
\]  
Define \[
J_{k}(t) = \int_{A(t)} \exp(-n V(|z|))  |z|^{2k} \, dz
\]  
for $0\le k\le n-1.$
Then, as \(n \to \infty\),  
\[
J_{n-k}(t) \; = \sqrt{2\pi a(t)} \frac{t_0^{2n}\,e^{-nV(t_0)} (1+o((\log n)^{-1}))}{(n\alpha\gamma_n)^{3/4}a^{2k}(t)(1+\frac{t}{\gamma_n}+\frac{2k}{t_0\sqrt{n\alpha\gamma_n}})^{3/2}}\exp(-\frac{{\gamma}_n}{2} (1+\frac{t }{{\gamma}_n})^2)
\]
uniformly on $|t|\lesssim (\log n)^{1/4}$ and $1\le k\lesssim \sqrt{n\log n}.$
\end{lem} 

\begin{proof}
Write \(z=x+iy\).  Then \(|z|^2=x^2+y^2\) and
\[
J_{n-k}(t)=\int_{x=a(t)}^{\infty}\int_{-\infty}^{\infty}
\exp\!\big(-nV(\sqrt{x^2+y^2})\big)\,(x^2+y^2)^{\,n-k}\,dy\,dx.
\]
For fixed \(x>a(t)\) set
\[
\phi_x(y)=V(\sqrt{x^2+y^2})-\frac{n-k}{n}\log(x^2+y^2)=\varphi_{u}(\sqrt{x^2+y^2})
\]
with $u=2k/n.$
The function \(\phi_x(y)\) is even in \(y\) and attains its minimum at \(y=0\) because $\phi_x'(0)=0$ and 
\[
\phi_x''(0)=\frac{1}{x}\varphi_u'(x)=\frac{1}{x}(V'(x)-\frac{2-u}{x})>0
\]
for \(x>a(t)\) (the inequality follows from \(xV'(x)>2\) and \(u=o(1)\)).  Laplace's method gives
\[
\int_{-\infty}^{\infty}e^{-n\phi_x(y)}\,dy = e^{-n\phi_x(0)}\sqrt{\frac{2\pi}{n\phi_x''(0)}}(1+O(n^{-1})).
\]
Hence
\[
\int_{-\infty}^{\infty}\exp\!\big(-nV(\sqrt{x^2+y^2})\big)(x^2+y^2)^{n-k}dy
=\exp(-n \varphi_u(x))\sqrt{\frac{2\pi x}{n\varphi_u'(x)}}.
\]
Substituting into the expression for \(J_{n-k}(t)\) we obtain
\begin{equation}\label{1}
J_{n-k}(t)=(1+O(n^{-1}))\int_{a(t)}^{\infty} \exp(-n \varphi_u(x))\sqrt{\frac{2\pi x}{n\varphi_u'(x)}}\,dx.
\end{equation}
For \(x>a(t)\) we have \(\varphi_u'(x)>0\) (strictly increasing), so the minimum of \(\varphi_u\) on \([a(t),\infty)\) occurs at the left endpoint \(x=a(t)\). The Laplace method gives 
\begin{equation}\label{jnklasttogo}\aligned
J_{n-k}(t)&=\frac{e^{-n \varphi_u(a(t))}(1+O(n^{-1}))}{n\varphi_u'(a(t))}\sqrt{\frac{2\pi a(t)}{n\varphi_u'(a(t))}}\\
&=(1+O(n^{-1})) \sqrt{2\pi a(t)}\,\frac{e^{-n \varphi_u(a(t))}}{(n\varphi_u'(a(t)))^{3/2}}. \endaligned 
\end{equation}
Similarly as for $\varphi_u(\widetilde{a}(t))$, we derive from Lemma \ref{lemphi} that 
$$\aligned n\varphi'_u({a}(t))&=\sqrt{\alpha n{\gamma}_n}(1+\frac{t}{{\gamma}_n}+\frac{2k}{t_0\sqrt{\alpha n {\gamma}_n}})(1+o((\log n)^{-1}))\endaligned $$ and 
$$\aligned n\varphi_u({a}(t)))&=n\varphi_0(t_0)+\frac{{\gamma}_n}{2} (1+\frac{t }{{\gamma}_n})^2+2k\log a(t)+ O(n^{-1/2}(\log n)^{3/2}). \endaligned$$
Inserting these two expressions into (\ref{jnklasttogo}), we have 
\[
J_{n-k}(t)=\sqrt{2\pi a(t)} \frac{t_0^{2n}\,e^{-nV(t_0)} (1+o((\log n)^{-1}))}{(n\alpha\gamma_n)^{3/4}a^{2k}(t)(1+\frac{t}{\gamma_n}+\frac{2k}{t_0\sqrt{n\alpha\gamma_n}})^{3/2}}\exp(-\frac{{\gamma}_n}{2} (1+\frac{t }{{\gamma}_n})^2).
\]
The proof is completed now. 
 \end{proof} 

\begin{lem}\label{decreasing} Both the ratio $Q_k(t):=\frac{J_{n-k}(t)}{J_{n-k}}$ and $\widetilde{Q}_k(t):=\frac{\widetilde{J}_{n-k}(t)}{J_{n-k}}$ are decreasing on $k$ for any $t\in \mathbb{R}.$ 	
\end{lem} 
\begin{proof} 
Taking account of \eqref{Jk} and \eqref{Jktilde}, we derive 
$$\widetilde{Q}_k(t)=\mathbb{P}(Y_{n-k+1}\ge  \widetilde{a}(t)),$$ where 
$Y_{j}$ is a random variable whose density function is proportional to $$\exp(-nV(r))r^{2j-1}1_{r>0}.$$ 
According to \cite[Lemma 2.4]{JQ19}, we see 
$$\mathbb{P}(Y_{n-k+1}\ge s)\le \mathbb{P}(Y_{n-k}\ge s)$$ for any $s\in\mathbb{R}$ and then $\widetilde{Q}_k(t)$ is decreasing on $k$.  
Meanwhile, 
using the polar coordinates, we understand the terms  $J_k(t)$ as 
$$J_{n-k}(t)=\int_{r\cos\theta>a(t), r>0, \theta\in [-\pi, \pi]} e^{-n V(r)}r^{2n-2k+1} dr d\theta.$$ 
Combining this with \eqref{Jk}, we derive that 
$$Q_k(t)=\mathbb{P}(Y_{n-k+1} \cos\Theta\ge a(t)).$$ Here, $\Theta$ is uniformly distributed on $[-\pi, \pi]$ and $Y_{n-k}$ is independent of $\Theta.$  
By conditioning on $\Theta$ and using the independence of $Y_{n-k}$ and $\Theta,$ we know $Q_k(t)$ is decreasing on $k.$ 	
\end{proof}

\section{Proof of Theorems} 
In this section, we give the proofs of Theorems, which consists of three subsections.
 
\subsection{Proof of Theorem \ref{main}} 
The proof of Theorem \ref{main} replies on the following two propositions, one is on the trace of $\mathbb{K}_n$ limited on $A(t)$ and $\widetilde{A}(t)$ and the other is for their Hilbert-Smith norms.  
\begin{prop}\label{keyprop} Let $\mathbb{K}_n$ and $A(t), \widetilde{A}(t)$ be defined as above. For sufficiently large $n,$ we have 
$${\rm Tr}(\mathbb{K}_n|_{A(t)})=(1+O((\log n)^{-1}))\exp(-t)(1-\frac{4t^2+20t-25\log\log n}{8\gamma_n}) $$ as well as $${\rm Tr}(\mathbb{K}_n|_{\widetilde{A}(t)})=(1+O((\log n)^{-1}))\exp(-t)(1-\frac{t^2+4t-4\log\log n}{2\widetilde{\gamma}_n})$$
uniformly on $|t|\lesssim (\log n)^{1/4}.$
\end{prop} 
\begin{proof}

Recall that  
\[
A(t)=\{z\in\mathbb{C}\mid \Re z\ge a(t)\},\qquad 
a(t)=t_0+\frac{\sqrt{\gamma_n}}{\sqrt{\alpha n}}+\frac{t}{\sqrt{\alpha n \gamma_n}}.
\]  
By definition,  
\[
\operatorname{Tr}(\mathbb{K}_n|_{A(t)})
=\int_{A(t)} K_n(z,z)\,dz
=\sum_{k=0}^{n-1}\frac{J_k(t)}{J_k}
=\sum_{k=1}^{n}Q_k(t).
\]  

We choose  
\(
j_n=\lfloor \sqrt{\gamma_n \alpha n t_0^2}\rfloor,
\) 
which is of order \(\sqrt{n\log n}\). Lemma~\ref{decreasing} gives  
\[
\sum_{k=1}^{j_n}Q_k(t)\le \sum_{k=1}^{n}Q_k(t)
=\sum_{k=1}^{j_n}Q_k(t)+\sum_{k=j_n+1}^{n}Q_k(t)
\le \sum_{k=1}^{j_n}Q_k(t)+nQ_{j_n}(t).
\]  
Thus it suffices to prove the two estimates  
\begin{equation}\label{trlim}
\sum_{k=1}^{j_n}Q_k(t)
=(1+O((\log n)^{-1}))\,e^{-t}
(1-\frac{t^2}{\gamma_n}-\frac{5t}{2\gamma_n}+\frac{25}{8\gamma_n}\log\log n)
\end{equation}  
and  
\begin{equation}\label{trremaindr}
\sum_{k=j_n+1}^{n}Q_k(t)\ll e^{-t}(\log n)^{-1}.
\end{equation}  

From Lemmas~\ref{lemjk} and~\ref{lemjkt}, we obtain the uniform asymptotic expression for \(1\le k\le j_n\):
\begin{equation}\label{summand}
Q_k(t)=\frac{
\exp(-\frac{\gamma_n}{2}(1+\frac{t}{\gamma_n})^2-\frac{2k^2}{\alpha n t_0^2})
}{
2\pi\sqrt{t_0}\,(n\alpha)^{1/4}\gamma_n^{3/4}
(1+\frac{t}{\gamma_n}+\frac{2k}{t_0\sqrt{\alpha n \gamma_n}})^{3/2}
}
(\frac{t_0}{a(t)})^{2k}.
\end{equation}  

We now concentrate on the sum of the factor  
\[
r_k(t):=
(1+\frac{t}{\gamma_n}+\frac{2k}{t_0\sqrt{\alpha n \gamma_n}})^{-3/2}
\exp(-\frac{2k^2}{\alpha n t_0^2})
(\frac{t_0}{a(t)})^{2k}.
\]  
To estimate this factor, observe that  
\[
\exp(-\frac{2k^2}{\alpha n t_0^2})
(1+\frac{\frac{2k}{t_0\sqrt{\alpha n \gamma_n}}}{1+\frac{t}{\gamma_n}})^{-3/2}
\le 1,
\]  
since both factors are at most \(1\). For a lower bound, using the elementary inequalities  
\[
e^{-x}\ge 1-x\quad (x\in \mathbb{R}),\qquad \log(1+x)\le x\quad (x>-1),
\]  
and the fact that \(t/\gamma_n=o(1)\) (which implies \(1+t/\gamma_n\sim 1\)), we get  
\[
\begin{aligned}
(1+\frac{\frac{2k}{t_0\sqrt{\alpha n \gamma_n}}}{1+\frac{t}{\gamma_n}})^{-3/2}
\exp(-\frac{2k^2}{\alpha n t_0^2}) 
&=\exp(
-\frac{3}{2}\log(1+\frac{\frac{2k}{t_0\sqrt{\alpha n \gamma_n}}}{1+\frac{t}{\gamma_n}})
-\frac{2k^2}{\alpha n t_0^2}
) \\
&\ge 1-\frac{3k}{t_0\sqrt{\alpha n \gamma_n}}-\frac{2k^2}{\alpha n t_0^2}.
\end{aligned}
\]    

Consequently, multiplying by \((t_0/a(t))^{2k}\) and recalling the prefactor \((1+t/\gamma_n)^{-3/2}\) (which comes from rewriting the denominator), we arrive at the sandwich bound  
\begin{equation}\label{twoside}
\begin{aligned}
(\frac{t_0}{a(t)})^{2k}
(1+\frac{t}{\gamma_n})^{-3/2}
(
1-\frac{3k}{t_0\sqrt{\alpha n \gamma_n}}
-\frac{2k^2}{\alpha n t_0^2}
) \le r_k(t) \le
(\frac{t_0}{a(t)})^{2k}
(1+\frac{t}{\gamma_n})^{-3/2}.
\end{aligned}
\end{equation}  

We now sum \eqref{twoside} over \(1\le k\le j_n\). Set  
\[
q:=(\frac{t_0}{a(t)})^2 \in (0,1).
\]  
Using the standard identities  
\[
\sum_{k=1}^{m}q^k=\frac{q(1-q^m)}{1-q},\qquad
\sum_{k=1}^{m} k^{\ell} q^k \sim \frac{q}{(1-q)^{\ell+1}}\quad (m\gg 1,\ \ell=1,2),
\]  
we obtain from the left-hand side of \eqref{twoside}
\[
\sum_{k=1}^{j_n} r_k(t)
\ge
(1+\frac{t}{\gamma_n})^{-3/2}
\frac{q}{1-q}
(
1-q^{j_n}
-\frac{3}{t_0\sqrt{\alpha n \gamma_n}\,(1-q)}
-\frac{2}{\alpha n t_0^2\,(1-q)^2}
),
\]  
and from the right-hand side  
\[
\sum_{k=1}^{j_n} r_k(t)
\le
(1+\frac{t}{\gamma_n})^{-3/2}
\frac{q}{1-q}
(1-q^{j_n}).
\]  

Now, using the definition of \(a(t)\) and the facts \(t=o(\gamma_n)\) and \(j_n=\lfloor \sqrt{\gamma_n\alpha n t_0^2}\rfloor\), we have  
\begin{equation}\label{1-q}
1-q=1-(\frac{t_0}{a(t)})^2
=\frac{(a(t)+t_0)(a(t)-t_0)}{a^2(t)}
=\frac{2\sqrt{\gamma_n}}{t_0\sqrt{\alpha n}}
(1+\frac{t}{\gamma_n})(1+o((\log n)^{-1})),
\end{equation}  
and  
\[
q^{j_n}=(\frac{t_0}{a(t)})^{2j_n}
=\exp(-t_0\sqrt{\alpha n\gamma_n}\,
\log(1+\frac{a(t)-t_0}{t_0}))
\sim \exp(-2\gamma_n-2t)
\ll \frac{1}{\log n}.
\]  
Thus \(1-q^{j_n}=1+o((\log n)^{-1})\), and because \(1-q \asymp \gamma_n^{1/2}/\sqrt{n}\), both correction terms  
\[
\frac{3}{t_0\sqrt{\alpha n \gamma_n}\,(1-q)}
\quad\text{and}\quad
\frac{2}{\alpha n t_0^2\,(1-q)^2}
\]  
are of order \(O(\gamma_n^{-1})=O((\log n)^{-1})\). Hence  
\[
1-q^{j_n}
-\frac{3}{t_0\sqrt{\alpha n \gamma_n}\,(1-q)}
-\frac{2}{\alpha n t_0^2\,(1-q)^2}
=1+O((\log n)^{-1}).
\]  
Combining this with \eqref{1-q}, we get  
\begin{equation}\label{sumofr}
\sum_{k=1}^{j_n} r_k(t)
=
(1+\frac{t}{\gamma_n})^{-3/2}
\frac{q}{1-q}
(1+O((\log n)^{-1})).
\end{equation}  

Since \(q=(\frac{t_0}{a(t)})^2=1+o((\log n)^{-1})\), inserting \eqref{1-q} into \eqref{sumofr} yields  
\[
\sum_{k=1}^{j_n} r_k(t)
=
(1+O((\log n)^{-1}))
(1+\frac{t}{\gamma_n})^{-5/2}
\frac{t_0\sqrt{n\alpha}}{2\sqrt{\gamma_n}}.
\]  

Multiplying this by the prefactor in \eqref{summand} (namely \(1/(2\pi\sqrt{t_0}(n\alpha)^{1/4}\gamma_n^{3/4})\)) and by the exponential term \(\exp(-\frac{\gamma_n}{2}(1+t/\gamma_n)^2)\), we obtain  
\begin{equation}\label{suminitial}
\sum_{k=1}^{j_n}Q_k(t)
=
(1+O((\log n)^{-1}))
\frac{\exp(-t-\frac{\gamma_n}{2}-\frac{t^2}{2\gamma_n})(\alpha n t_0^2)^{1/4}}{4\pi\,\gamma_n^{5/4}}
(1+\frac{t}{\gamma_n})^{-5/2}.
\end{equation}  

Finally, using the definition  
\[
\gamma_n=\frac12\log(\frac{n\alpha t_0^2}{8\pi^4(\log n)^5}),
\]  
we have  
\begin{equation}\label{expgamma}
\exp(-\frac{\gamma_n}{2})
=
\frac{2^{3/4}\pi(\log n)^{5/4}}{(n\alpha t_0^2)^{1/4}}.
\end{equation}  
Substituting \eqref{expgamma} into \eqref{suminitial} gives, uniformly for \(|t|\lesssim (\log n)^{1/4}\),  
\[
\begin{aligned}
\sum_{k=1}^{j_n}Q_k(t)
&=
(1+O((\log n)^{-1}))
\exp(-t-\frac{t^2}{2\gamma_n})
(1+\frac{t}{\gamma_n})^{-5/2}
(\frac{\log n}{2\gamma_n})^{5/4} \\
&=
(1+O((\log n)^{-1}))
e^{-t}
(
1-\frac{t^2}{2\gamma_n}-\frac{5t}{2\gamma_n}
+\frac{5}{8\gamma_n}\log(\frac{8\pi^4(\log n)^5}{\alpha t_0^2})
) \\
&=
(1+O((\log n)^{-1}))
e^{-t}
(
1-\frac{t^2}{2\gamma_n}-\frac{5t}{2\gamma_n}
+\frac{25}{8\gamma_n}\log\log n
),
\end{aligned}
\]  
where in the last step the constant term \(\frac{5}{8\gamma_n}\log(8\pi^4/(\alpha t_0^2))\) has been absorbed into the \(O((\log n)^{-1})\) error.

 Next, we prove \eqref{trremaindr}. Lemma \ref{decreasing} implies that 
 $$\sum_{j=j_n+1}^n Q_k(t)\le n Q_{j_n}(t). $$
 The asymptotic \eqref{summand} and \eqref{expgamma} gives 
  \begin{equation}\label{nQjn}\aligned n Q_{j_n}(t)&\asymp e^{-3t}\sqrt{n\log n} \exp(-4\gamma_n)\ll e^{-t}(\log n)^{-1}. 
 \endaligned \end{equation}
 Now we work on ${\rm Tr}(\mathbb{K}_n|_{\widetilde{A}(t)}).$ 
 Similarly, choosing $i_n=\lfloor\sqrt{\alpha n\widetilde{\gamma}_n t_0^2}\rfloor,$ we have 
 $${\rm Tr}(\mathbb{K}_n|_{\widetilde{A}(t)})=\sum_{k=1}^{n}\widetilde{Q}_k(t)=\sum_{k=1}^{i_n}\widetilde{Q}_k(t)+\sum_{k=i_n+1}^n \widetilde{Q}_k(t)$$ 
 and then 
 \begin{equation}\label{trtilde}
 	\sum_{k=1}^{i_n}\widetilde{Q}_k(t)\le {\rm Tr}(\mathbb{K}_n|_{\widetilde{A}(t)})=\sum_{k=1}^{n}\widetilde{Q}_k(t)\le \sum_{k=1}^{i_n}\widetilde{Q}_k(t)+n\widetilde{Q}_{i_n}(t).
 \end{equation}
 Lemmas \ref{lemjk} and \ref{lemjktilde} yield 
\begin{equation}\label{quotientilde}  
\widetilde{Q}_k(t)=\frac{(1+o((\log n)^{-1}))}{\sqrt{2\pi\widetilde{\gamma}_n}(1+\frac{t}{\widetilde{\gamma}_n}+\frac{2k}{t_0\sqrt{\alpha n \widetilde{\gamma}_n}})}\exp(-\frac{\widetilde{\gamma}_n}{2} (1+\frac{t }{\widetilde{\gamma}_n})^2-\frac{2k^2}{\alpha n t_0^2})(\frac{t_0}{\widetilde{a}(t)})^{2k}
\end{equation}
uniformly on $1\le k\le i_n.$ 
In analogy with the treatment of \eqref{suminitial}, the factor $$(1+\frac{t}{\widetilde{\gamma}_n}+\frac{2k}{t_0\sqrt{\alpha n \widetilde{\gamma}_n}})^{-1}\exp(-\frac{2k^2}{\alpha n t_0^2})$$ in the summand can be reduced to \((1+\frac{t}{\widetilde{\gamma}_n})^{-1}\), since the error incurred by this replacement is uniformly absorbed into the \(O((\log n)^{-1})\) term. Therefore,
\begin{equation}\label{suminitialtilde}
\begin{aligned} 
\sum_{k=1}^{i_n}\widetilde{Q}_k(t)&=(1+O((\log n)^{-1}))\frac{ \exp(-t-\frac{\widetilde{\gamma}_n}{2}-\frac{t^2}{2\widetilde{\gamma}_n}) }{\sqrt{2\pi\widetilde{\gamma}_n}}(1+\frac{t}{\widetilde{\gamma}_n})^{-1}\sum_{k=1}^{i_n}(\frac{t_0}{\widetilde{a}(t)})^{2k} \\
&=(1+O((\log n)^{-1}))\frac{ \sqrt{t_0^2\alpha n}\exp(-t-\frac{\widetilde{\gamma}_n}{2}-\frac{t^2}{2\widetilde{\gamma}_n}) }{2\sqrt{2\pi}\widetilde{\gamma}_n}(1+\frac{t}{\widetilde{\gamma}_n})^{-2}.
\end{aligned}
\end{equation}
Here, we use the fact $$(\frac{t_0}{\widetilde{a}(t)})^{2i_n}\asymp\exp(-2\widetilde{\gamma}_n-2t)\ll n^{-1}.$$ 
Now $\widetilde{\gamma}_n=\log\frac{n\alpha t_0^2}{8\pi (\log n)^2},$ whence    
\begin{equation}\label{expgammatilde}\exp(-\frac{\widetilde{\gamma}_n}{2})=\frac{2\sqrt{2\pi}\log n}{\sqrt{n\alpha t_0^2}}.\end{equation}
Plugging \eqref{expgammatilde} into \eqref{suminitialtilde}, we derive 
\begin{equation}\label{partialsumtilde}\aligned\sum_{k=1}^{i_n}\widetilde{Q}_k(t)&=(1+O((\log n)^{-1}))\exp(-t-\frac{t^2}{2\widetilde{\gamma}_n})(1+\frac{t}{\widetilde{\gamma}_n})^{-2}\frac{\log n}{\widetilde{\gamma}_n}\\
&=(1+O((\log n)^{-1}))\exp(-t)(1-\frac{t^2}{2\widetilde{\gamma}_n}-\frac{2t}{\widetilde{\gamma}_n}+\frac{2\log\log n}{\widetilde{\gamma}_n})\endaligned \end{equation}
uniformly on $|t|\lesssim (\log n)^{1/4}.$ The proof is completed now. 
Similarly as for $nQ_{j_n}(t),$ we have 
\begin{equation}\label{nQin}n\widetilde{Q}_{i_n}(t)\asymp e^{-3t}\sqrt{n\log n} \exp(-4\widetilde{\gamma}_n)\ll e^{-t}(\log n)^{-1}.\end{equation}
This together with \eqref{partialsumtilde} completes the proof. 
 \end{proof}  
 
\begin{prop}\label{HS}
	With $\mathbb{K}_n$, $A(t)$ and $\widetilde{A}(t)$ defined as above, we have
	\[
	\|\mathbb{K}_n|_{A(t)}\|_2 \lesssim e^{-t}(\log n)^{1/4}n^{-\frac{1}{20}}
	\quad\text{and}\quad
	\|\mathbb{K}_n|_{\widetilde{A}(t)}\|_2 \lesssim \frac{e^{-t}\sqrt{\log n}}{n^{1/4}}
	\]
	uniformly for $|t|\lesssim (\log n)^{1/4}.$
\end{prop}

\begin{proof}
	
By definition,
\[
\|\mathbb{K}_n|_{A(t)}\|_2^2=\int_{A(t)\times A(t)}|K_n(z,w)|^2\,dz\,dw.
\]
For simplicity, set
\[
M_{j,k}(t)=\frac{1}{\sqrt{J_{n-j}J_{n-k}}}
\int_{A(t)} \exp(-nV(|z|))\,z^{n-k}\bar{z}^{n-j}\,dz.
\]
Clearly, \(M_{k,k}(t)=Q_k(t)\). A straightforward expansion gives
\[
\|\mathbb{K}_n|_{A(t)}\|_2^2=\sum_{j,k=1}^n M_{j,k}^2(t).
\]
By the Cauchy--Schwarz inequality, these quantities satisfy
\begin{equation}\label{CS}
M_{j,k}^2(t)\le M_{j,j}(t)M_{k,k}(t)=Q_k(t)Q_j(t).
\end{equation}
We split the double sum into three regions:
\(1\le j,k\le j_n\);
\(1\le k\le j_n<j\le n\);
and \(j_n+1\le j,k\le n\).
From \eqref{CS} and \eqref{nQjn}, we obtain
\begin{equation}\label{HS1}
\sum_{j_n+1\le j,k\le n} M_{j,k}^2(t)
\le (\sum_{j=j_n+1}^n Q_j(t))^2 \lesssim e^{-6t-8\gamma_n},
\end{equation}
and similarly,
\begin{equation}\label{HS2}
\sum_{1\le k\le j_n < j\le n} M_{j,k}^2(t)
\le \sum_{j=j_n+1}^n Q_j(t)\sum_{k=1}^{n} Q_k(t)
\lesssim e^{-4t-4\gamma_n}.
\end{equation}

It remains to estimate the most involved part, \(\sum_{1\le j,k\le j_n}M_{j,k}^2(t)\). By symmetry,
\[
\sum_{1\le j,k\le j_n}M_{j,k}^2(t)
\lesssim \sum_{1\le j\le k\le j_n}M_{j,k}^2(t)
\le \sum_{j=1}^{j_n}\sum_{d=0}^{j_n-1}M_{j,j+d}^2(t).
\]
Set \(s_n=[n^{2/5}]\). Using \eqref{CS} and the monotonicity of \(Q_j\),
\[
\sum_{j=1}^{j_n}\sum_{d=0}^{s_n}M_{j,j+d}^2(t)
\le \sum_{j=1}^{j_n}\sum_{d=0}^{s_n}Q_j(t)Q_{j+d}(t)
\le \sum_{j=1}^{n}Q_j(t)\sum_{d=0}^{s_n}Q_{1+d}(t).
\]
Indeed, using the asymptotics \eqref{summand} and \eqref{expgamma}, and proceeding as in the derivation of \eqref{suminitial}, we get
\[
\sum_{d=0}^{s_n}Q_{1+d}(t)
\asymp e^{-t}\sqrt{\log n}\,n^{-1/2}
\sum_{k=1}^{s_n}(\frac{t_0}{a(t)})^{2k}.
\]
For this geometric sum,
\[
\sum_{k=1}^{s_n}(\frac{t_0}{a(t)})^{2k}
=
\frac{t_0^2(1-(\frac{t_0}{a(t)})^{2s_n})}{a^2(t)-t_0^2}
\lesssim
\sqrt{\frac{n}{\log n}}
(
1-(\frac{t_0}{\widetilde{a}(t)})^{2s_n}
).
\]
Now \(s_n=n^{2/5}+O(1)\), and it holds that \(s_n(\widetilde{a}(t)-t_0)\sim n^{-1/10}\sqrt{\log n}=o(1)\), whence
\[
1-(\frac{t_0}{\widetilde{a}(t)})^{2s_n}\sim 2s_n \frac{\widetilde{a}(t)-t_0}{t_0}
\asymp \frac{\sqrt{\log n}}{n^{1/10}}.
\]
Consequently,
\begin{equation}\label{HS3}
\sum_{j=1}^{j_n}\sum_{d=0}^{s_n}M_{j,j+d}^2(t)
\lesssim e^{-2t}\frac{\sqrt{\log n}}{n^{1/10}}.
\end{equation}

Recall that for \(k\neq j\),
\[
\int_{\{r\cos\theta>a,\ \theta\in[-\pi,\pi]\}}
\cos((k-j)\theta)\,d\theta
=
\frac{2\sin((k-j)\arccos(a/r))}{k-j}.
\]
Using polar coordinates, we obtain
\[
\begin{aligned}
M_{j,k}(t)&=
\frac{2}{(k-j)\sqrt{J_{n-j}J_{n-k}}}
\int_{a(t)}^{\infty} e^{-nV(r)} r^{2n-k-j+1}
\sin((k-j)\arccos(\frac{a(t)}{r}))\,dr\\
&\lesssim  \frac{1}{(k-j)\sqrt{J_{n-j}J_{n-k}}}\int_{a(t)}^{\infty} e^{-nV(r)} r^{2n-k-j+1}\,dr.
\end{aligned}
\]
A calculation analogous to that in Lemma \ref{lemjkt} yields
$$\int_{a(t)}^{\infty} e^{-nV(r)} r^{2n-k-j+1}\,dr\asymp \frac{t_0^{2n}\exp\{-nV(t_0)-\frac{\gamma_n}{2}(1-\frac{t}{\gamma_n})^2\}}{\sqrt{\alpha n\gamma_n}(a(t))^{k+j-1}} .$$
Then Lemma \ref{lemjk} gives 
$$\sqrt{J_{n-j}J_{n-k}}\asymp \frac{t_0^{2n-(j+k)+1}}{\sqrt{n}}\exp\{nV(t_0)+\frac{j^2+k^2}{\alpha nt_0^2}\} $$
Combining this with \ref{expgamma}, we obtain the upper bound estimate 
\[
M_{j,k}(t)\lesssim
\frac{(\log n)^{3/4}\, e^{-t}}{(k-j) n^{1/4}}
(\frac{t_0}{a(t)})^{k+j},
\]
and hence
\[
M_{j,j+d}^2(t)\le \frac{e^{-2t}(\log n)^{3/2}}{\sqrt{n} d^2} (\frac{t_0}{a(t)})^{4j}.
\]
Therefore,
\begin{equation}\label{HS4}
\begin{aligned}
\sum_{j=1}^{j_n}\sum_{d=s_n+1}^{j_n} M_{j,j+d}^2(t)
&\lesssim
\frac{e^{-2t}(\log n)^{3/2}}{\sqrt{n}}
\sum_{j=1}^{+\infty}
(\frac{t_0}{a(t)})^{4j}\sum_{d=s_n+1}^{+\infty} d^{-2}\\
&\asymp\frac{e^{-2t}(\log n)^{3/2}}{s_n \sqrt{n}}\times \frac{\sqrt{n}}{\sqrt{\log n}}\asymp e^{-2t}n^{-2/5}\log n.
\end{aligned}
\end{equation}
Here we used the sum
\[
\sum_{j=1}^{+\infty}
(\frac{t_0}{a(t)})^{4j}
\asymp
(a(t)-t_0)^{-1}
\asymp
\sqrt{n} (\log n)^{-1/2}.
\]
Combining \eqref{HS1}--\eqref{HS4}, we obtain
\[
\|\mathbb{K}_n|_{A(t)}\|_2\lesssim e^{-t}(\log n)^{1/4}n^{-1/20}.
\]

For the spectral radius case, we define analogously
\[
\widetilde{M}_{j,k}(t)=
\frac{1}{\sqrt{J_{n-j}J_{n-k}}}
\int_{\widetilde{A}(t)} e^{-nV(|z|)} z^{n-k}\bar{z}^{n-j}\,dz.
\]
Writing \(z=re^{i\theta}\), for \(j\neq k\) we have
\[
\widetilde{M}_{j,k}(t)=
\frac{1}{\sqrt{J_{n-j}J_{n-k}}}
\int_{\widetilde{a}(t)}^{\infty} e^{-nV(r)}r^{2n-k-j+1}\,dr
\int_0^{2\pi} e^{i(j-k)\theta}\,d\theta=0.
\]
Consequently,
\[
\|\mathbb{K}_n|_{\widetilde{A}(t)}\|_2^2
=
\sum_{j=1}^n \widetilde{Q}_j^2(t)
\le
\sum_{j=1}^{i_n}\widetilde{Q}_j^2(t)+n\widetilde{Q}_{i_n}^2(t).
\]
Using \eqref{suminitialtilde}, we get
\[
\sum_{j=1}^{i_n}\widetilde{Q}_j^2(t)
\lesssim
\frac{e^{-2t-\widetilde{\gamma}_n}}{\widetilde{\gamma}_n}
\sum_{j=1}^{i_n}(\frac{t_0}{\widetilde{a}(t)})^{4j}
\asymp 
\frac{e^{-2t} \log n}{\sqrt{n}}.
\]
Also, from \eqref{nQin} we have
\[
n\widetilde{Q}_{i_n}^2(t)
\asymp e^{-6t}\log n\, e^{-8\widetilde{\gamma}_n}
\ll \frac{e^{-2t} \log n}{\sqrt{n}}.
\]
Therefore,
\[
\|\mathbb{K}_n|_{\widetilde{A}(t)}\|_2 \lesssim \frac{e^{-t}\sqrt{\log n}}{n^{1/4}}.
\]
\end{proof}

{\bf Proof of Theorem \ref{main}}. 
We first obtain from the triangle inequality and the hole probability formula that 
$$\aligned|\mathbb{P}(X_n\le t)-e^{-e^{-t}}|&\le |\exp(-{\rm Tr} (\mathbb{K}_n|_{A(t)}))-e^{-e^{-t}}|\\
&+|\det({\rm I}-\mathbb{K}_n|_{A(t)})-\exp(-{\rm Tr} (\mathbb{K}_n|_{A(t)}))|. 
\endaligned $$
and 
$$\aligned|\mathbb{P}(X_n\le t)-e^{-e^{-t}}|&\ge |\exp(-{\rm Tr} (\mathbb{K}_n|_{A(t)}))-e^{-e^{-t}}|\\
&\quad-|\det({\rm I}-\mathbb{K}_n|_{A(t)})-\exp(-{\rm Tr} (\mathbb{K}_n|_{A(t)}))| \endaligned $$
The inequality \eqref{comp} and Proposition \ref{HS} yield 
\begin{equation}\label{ignore}\aligned
|\det({\rm I}-\mathbb{K}_n|_{A(t)})-\exp(-{\rm Tr} (\mathbb{K}_n|_{A(t)}))|	&\lesssim e^{-t}(\log n)^{1/4}n^{-\frac{1}{20}}\exp(-{\rm Tr} (\mathbb{K}_n|_{A(t)}))\\
&\lesssim e^{-t-e^{-t}}(\log n)^{1/4}n^{-\frac{1}{20}}.
\endaligned 
\end{equation}
The Proposition \ref{keyprop} tells us 
$$e^{-t}-{\rm Tr} (\mathbb{K}_n|_{A(t)})=e^{-t}\frac{8t^2+20t-25\log\log n+O(1)}{8\gamma_n},$$ 
which is uniformly $o(1)$ for $-\frac{1}{4}\log \log n\le t\le (\log n)^{1/4}.$ Thereby, 
\begin{equation}\label{keyintegre}\aligned |e^{-{\rm Tr} (\mathbb{K}_n|_{A(t)})}-e^{-e^{-t}}|&=e^{-e^{-t}}|\exp(e^{-t}-{\rm Tr} (\mathbb{K}_n|_{A(t)}))-1|\\
&=e^{-e^{-t}-t} \frac{|8t^2+20t-25\log\log n+O(1)|}{8\gamma_n}.
\endaligned \end{equation}
The expressions \eqref{ignore} and \eqref{keyintegre} guarantee that 
\begin{equation}\label{maincontri}|\mathbb{P}(X_n\le t)-e^{-e^{-t}}|= e^{-e^{-t}-t} \frac{1+o(1)}{8\gamma_n}|8t^2+20t-25\log\log n+O(1)|
	\end{equation}
uniformly on $[-\ell_1(n), \ell_2(n)].$

The functions $e^{-e^{-t}-t} t^{k}$ is bounded for $k=0, 1, 2$ ensures that 
$$\aligned \sup_{-\ell_1(n)\le t\le \ell_2(n)} |\mathbb{P}(X_n\le t)-e^{-e^{-t}}|&=\frac{25\log \log n(1+o(1))}{8\gamma_n} \sup_{-\ell_1(n)\le t\le \ell_2(n)}e^{-e^{-t}-t}\\
&=\frac{25\log\log n}{4e \log n}(1+o(1)). \endaligned $$ 
Here, for the last equality, we use the fact that 
$$2\gamma_n=\log n(1+o(1)).$$ 

On the other hand, we have from the triangle inequality and the monotonicity of the distribution function that 
$$\aligned \sup_{x\le -\ell_1(n)}|\mathbb{P}(X_n\le t)-e^{-e^{-t}}|&\le \mathbb{P}(X_n\le -\ell_1(n))+e^{-e^{\ell_1(n)}}\lesssim e^{-e^{\ell_1(n)}}\ll \frac{1}{\log n}. \endaligned $$
As a parallel result, it holds 
$$\aligned \sup_{x\ge \ell_2(n)}|\mathbb{P}(X_n\le t)-e^{-e^{-t}}|&\le \mathbb{P}(X_n\ge \ell_2(n))+e^{-e^{-\ell_2(n)}}\lesssim e^{-e^{-\ell_2(n)}}\ll \frac{1}{\log n}. \endaligned $$ 
Combining these three supremums together, we get the desired Berry-Esseen bound for the rightmost eigenvalue.
Similarly, 
$$e^{-t}-{\rm Tr} (\mathbb{K}_n|_{\widetilde{A}(t)})=e^{-t}\frac{t^2+4t-4\log\log n+O(1)}{2\widetilde{\gamma}_n},$$ 
which is uniformly $o(1)$ for $-\ell_1(n)\le t\le \ell_2(n).$ Thereby, 
$$\aligned |e^{-{\rm Tr} (\mathbb{K}_n|_{\widetilde{A}(t)})}-e^{-e^{-t}}|&=e^{-e^{-t}}|\exp(e^{-t}-{\rm Tr} (\mathbb{K}_n|_{\widetilde{A}(t)}))-1|\\
&=e^{-e^{-t}-t} \frac{|t^2+4t-4\log\log n+O(1)|}{2\widetilde{\gamma}_n}.
\endaligned $$
We derive from Proposition \ref{HS} that 
$$|\det({\rm I}-\mathbb{K}_n|_{\widetilde{A}(t)})-\exp(-{\rm Tr} (\mathbb{K}_n|_{\widetilde{A}(t)}))|\ll |e^{-{\rm Tr} (\mathbb{K}_n|_{\widetilde{A}(t)})}-e^{-e^{-t}}|.$$ 
Therefore, similarly it follows that  
$$\aligned \sup_{-\ell_1(n)\le t\le \ell_2(n)} |\mathbb{P}(\widetilde{X}_n\le t)-e^{-e^{-t}}|&=(1+o(1))\sup_{-\ell_1(n)\le t\le \ell_2(n)} |e^{-{\rm Tr} (\mathbb{K}_n|_{\widetilde{A}(t)})}-e^{-e^{-t}}|\\
&=\frac{2\log \log n(1+o(1))}{\widetilde{\gamma}_n} \sup_{-\ell_1(n)\le t\le \ell_2(n)}e^{-e^{-t}-t}\\
&=\frac{2\log\log n}{e \log n}(1+o(1)). \endaligned $$ 
 The same analysis yields the supremum over two side intervals is negligible and then  
the corresponding proof for the spectral radius is completed. 

Now we provide the proof of Theorem \ref{thmrealpart}.

\subsection{Proof of Theorem \ref{thmrealpart}}

As explained in the introduction, to study the large deviations of \(\max_{1\le i\le n}\Re \sigma_i\), it suffices to consider the trace of \(\mathbb{K}_n\) restricted to the half-plane \(\{\Re z>t\}\) for \(t>t_0\).

Define
\[
R_k(t):=\int_{\Re z\ge t} \exp(-nV(|z|))\,|z|^{2k}\,d^2z,\qquad 0\le k\le n-1.
\]
Then
\[
\operatorname{Tr}(\mathbb{K}_n|_{\{\Re z>t\}})
=\int_{\Re z\ge t}K_n(z,z)\,dz
=\sum_{k=1}^{n}\frac{R_{n-k}(t)}{J_{n-k}}.
\]
Recall that \(\varphi_u(r)=V(r)-(2-u)\log r\) for \(0\le u<2\). Proceeding analogously to the derivation of \eqref{jnklasttogo} and taking \(u=2k/n\), we obtain
\begin{equation}
\begin{aligned}
R_{n-k}(t)
&=(1+O(n^{-1})) \sqrt{2\pi t}\,\frac{e^{-n \varphi_u(t)}}{n^{3/2}(\varphi_u'(t))^{3/2}}\\
&=(1+O(n^{-1})) \sqrt{2\pi }\,e^{-nV(t)}t^{2n-2k+2}n^{-3/2}(tV'(t)-2+\frac {2k}n)^{-3/2}\\
&\sim \sqrt{2\pi }\,e^{-nV(t)}t^{2n-2k+2}n^{-3/2}(tV'(t)-2)^{-3/2},
\end{aligned}
\end{equation}
uniformly for \(t>t_0\) and \(1\le k\le j_n\). Combining this with Lemma~\ref{lemjk} yields
\begin{equation}\label{quotlarge}
\frac{R_{n-k}(t)}{J_{n-k}}
\sim \frac{\sqrt{\alpha }\,t}{2\pi  n (tV'(t)-2)^{3/2}}
\exp(-nV(t)+nV(t_0)-\frac{2k^2}{\alpha n t_0^2})\,(\frac{t}{t_0})^{2n-2k+1}.
\end{equation}

When summing over \(1\le k\le j_n\), the factor \(\exp(-\frac{2k^2}{\alpha n t_0^2})\) does not affect the leading order, just as in the derivation of \eqref{suminitial}. Hence
\begin{equation}\label{sumquotlarge}
\begin{aligned}
\sum_{k=1}^{j_n}\frac{R_{n-k}(t)}{J_{n-k}}
&\sim \frac{\sqrt{\alpha }\,t}{2\pi  n (tV'(t)-2)^{3/2}}
\exp(-nV(t)+nV(t_0))\,(\frac{t}{t_0})^{2n}
\sum_{k=1}^{j_n}(\frac{t_0}{t})^{2k-1}\\
&\sim \frac{\sqrt{\alpha }\,t^2 t_0}{2\pi  n(t^2-t_0^2) (tV'(t)-2)^{3/2}}
\exp(-nV(t)+nV(t_0))\,(\frac{t}{t_0})^{2n}.
\end{aligned}
\end{equation}

Comparing \eqref{quotlarge} with \eqref{sumquotlarge}, and using the fact
\[
n \exp(-\frac{2j_n^2}{\alpha n t_0^2})(\frac{t_0}{t})^{2j_n}\ll 1
\]
together with the monotonicity of \(\frac{L_{n-k}(t)}{J_{n-k}}\), we obtain
\[
\sum_{k=1}^{n}\frac{R_{n-k}(t)}{J_{n-k}}
\sim \sum_{k=1}^{j_n}\frac{R_{n-k}(t)}{J_{n-k}}
\sim \frac{\sqrt{\alpha }\,t^2 t_0}{2\pi  n(t^2-t_0^2) (tV'(t)-2)^{3/2}}
\exp(-nV(t)+nV(t_0))\,(\frac{t}{t_0})^{2n}.
\]

Set
\[
g(t):=V(t)-V(t_0)-2\log \frac{t}{t_0}.
\]
Since \(t_0\) is the unique minimizer of \(V(t)-2\log t\), we have \(g(t)>g(t_0)=0\). This ensures that
\[
\operatorname{Tr}(\mathbb{K}_n|_{\{\Re z>t\}})=\sum_{k=1}^{n}\frac{R_{n-k}(t)}{J_{n-k}}
\asymp \frac1n \exp(-n g(t))=o(1).
\]
Therefore,
\[
\det({\rm I}-\mathbb{K}_n|_{\{\Re z>t\}})
\sim \exp(-\operatorname{Tr}(\mathbb{K}_n|_{\{\Re z>t\}}))
\sim 1-\operatorname{Tr}(\mathbb{K}_n|_{\{\Re z>t\}}).
\]
Consequently, \(\|\mathbb{K}_n|_{\{\Re z>t\}}\|_2\ll 1\), and furthermore
\[
\mathbb{P}(\max_i\Re\sigma_i>t)
=1-\det({\rm I}-\mathbb{K}_n|_{\{\Re z>t\}})
\sim \operatorname{Tr}(\mathbb{K}_n|_{\{\Re z>t\}}).
\]
This completes the large deviation part for fixed \(t>t_0\).

We now consider the scaling \(t=t_0+ d_n\), with \(d_n\to 0\). Proceeding analogously to the fixed-\(t\) case, we obtain
\[
\begin{aligned}
\operatorname{Tr}(\mathbb{K}_n|_{\{\Re z>t_0+d_n\}})
&\sim \frac{\sqrt{\alpha}\,t_0^2}{4\pi  n\,  d_n\,((t_0+d_n)V'(t_0+d_n)-2)^{3/2}} \\
&\quad\times \exp(-nV(t_0+ d_n)+nV(t_0))\,(\frac{t_0+ d_n}{t_0})^{2n}.
\end{aligned}
\]
Using Taylor expansion with \(V'(t_0)=2/t_0\) and \(\alpha= V''(t_0)+2/t_0^2\), we get
\begin{equation}\label{modeV}
V(t_0+ d_n)-V(t_0)-2\log(1+ \frac{ d_n}{t_0})
=\frac{\alpha}{2} d_n^2+O(|d_n|^3),
\end{equation}
and
\begin{equation}\label{demoV}
(t_0+d_n)V'(t_0+ d_n)-2\sim \alpha  t_0 d_n.
\end{equation}
Combining these estimates yields the order
\[
\operatorname{Tr}(\mathbb{K}_n|_{\{\Re z>t_0+  d_n\}})
\asymp n^{-1} d_n^{-5/2}\exp(-\frac{\alpha}{2} n d_n^2).
\]
In particular, taking \(d_n=\frac{t\sqrt{\gamma_n}}{\sqrt{n\alpha}}\) with $t>1$, we have
\begin{equation}\label{Tro1}
n^{-1} d_n^{-5/2}\exp(-\frac{\alpha}{2} t^2 d_n^2)
\asymp (n(\log n)^5)^{\frac{1-t^2}{4}}=o(1).
\end{equation}
More generally, this implies that
\[
\operatorname{Tr}(\mathbb{K}_n|_{\{\Re z>t_0+ d_n\}})=o(1)
\]
whenever \(d_n\gg \frac{\sqrt{\gamma_n}}{\sqrt{n\alpha}}\). Hence, under this condition, the determinant approximation applies and we obtain the probability asymptotics
\begin{equation}\label{modefinal}
\begin{aligned}
\mathbb{P}(\max\Re\sigma_i\ge t_0+d_n)
&\sim \operatorname{Tr}(\mathbb{K}_n|_{\{\Re z>t_0+d_n\}}) \\
&\sim \frac{\sqrt{t_0}}{4\pi \alpha n d_n^{5/2}}
\exp(-nV(t_0+d_n)+nV(t_0))\,(\frac{t_0+d_n}{t_0})^{2n}.
\end{aligned}
\end{equation}

In the moderate deviation regime \(\sqrt{\log n}/\sqrt{n}\ll d_n\ll n^{-1/3}\), it follows from \eqref{modeV} that
\[
nV(t_0+ d_n)-nV(t_0)-2n\log(1+ \frac{d_n}{t_0})
=\frac{n\alpha}{2}d_n^2+o(1),
\]
so the asymptotic simplifies to
\[
\mathbb{P}(\max\Re\sigma_i\ge t_0+d_n)
\sim \frac{\sqrt{t_0}}{4\pi \alpha n d_n^{5/2}}
\exp(-\frac{\alpha  n d_n^2}{2}).
\]
For the specific choice \(d_n=\frac{t\sqrt{\gamma_n}}{\sqrt{\alpha n}}\) with $t>1$, we obtain
\begin{equation}\label{smalltolast}
\begin{aligned}
\mathbb{P}(\max\Re\sigma_i\ge t_0+t\frac{\sqrt{\gamma_n}}{\sqrt{\alpha n}})
\sim \frac{(\alpha n t_0^2)^{1/4}}{4\pi (t^2 \gamma_n)^{5/4}}
\exp(-\frac{t^2\gamma_n}{2}),
\end{aligned}
\end{equation}
and with \(\gamma_n=\frac12\log\frac{\alpha n t_0^2}{8\pi^4(\log n)^5}\), this becomes
\[
\mathbb{P}(\max\Re\sigma_i\ge t_0+t\frac{\sqrt{\gamma_n}}{\sqrt{\alpha n}})
\sim t^{-5/2}(\frac{8\pi^4 (\log n)^5}{n \alpha t_0^2})^{\frac{t^2-1}{4}}.
\]

Finally, consider deviations of the form
\[
\mathbb{P}(\max_{1\le i\le n}\Re \sigma_i\ge t_0+\frac{\sqrt{\gamma_n}}{\sqrt{\alpha n}}+\frac{ v_n}{\sqrt{\alpha n \gamma_n}})
\qquad (1\ll v_n\ll \log n).
\]
The general formula \eqref{modefinal} remains valid with \[
d_n=\frac{\sqrt{\gamma_n}}{\sqrt{\alpha n}}+\frac{ v_n}{\sqrt{\alpha n \gamma_n}}.
\]
Following the same computation as in \eqref{smalltolast}, we get
\[
\begin{aligned}
&\mathbb{P}(\max\Re\sigma_i\ge t_0+\frac{\sqrt{\gamma_n}}{\sqrt{\alpha n}}+\frac{ v_n}{\sqrt{\alpha n \gamma_n}})\\
&\sim \frac{(\alpha n t_0^2)^{1/4}}{4\pi \gamma_n^{5/4}}
\exp(-\frac{\gamma_n}{2}-v_n-\frac{v_n^2}{2\gamma_n})
=\exp(-v_n-\frac{v_n^2}{2\gamma_n}).
\end{aligned}
\]

We now turn to the spectral radius.

\subsection{Proof of Theorem \ref{thmsp}}

The proof for the spectral radius follows the same line of reasoning. For completeness, we state the key steps.

For \(t>t_0\), define
\[
\widetilde{R}_{k}(t)=\int_{|z|\ge t} \exp(-nV(|z|))\,|z|^{2k}\,d^2z,\qquad 0\le k\le n-1,
\]
and set \(i_n=\lfloor \alpha n \widetilde{\gamma}_n t_0^2\rfloor\). Analogously to the rightmost eigenvalue case, we have the basic correspondence
\begin{equation}\label{connection}
\mathbb{P}(\max_{1\le i\le n}|\sigma_i|>t)
\sim \operatorname{Tr}(\mathbb{K}_n|_{\{|z|>t\}})
\sim \sum_{k=1}^{i_n} \frac{\widetilde{R}_{n-k}(t)}{J_{n-k}}.
\end{equation}
The same saddle-point analysis yields, uniformly for \(1\le k\le i_n\),
\[
\widetilde{R}_{n-k}(t)\sim 2\pi e^{-n\varphi_u(t)}\frac{1}{n\varphi_u'(t)}
\sim \frac{2\pi e^{-nV(t)}t^{2n-2k+1}}{n(V'(t)-\frac{2}{t})},
\]
and thus
\[
\frac{\widetilde{R}_{n-k}(t)}{J_{n-k}}
\sim \frac{\sqrt{\alpha}}{\sqrt{2\pi n}(V'(t)-\frac 2t)}
\exp(-nV(t)+nV(t_0)-\frac{2k^2}{\alpha n t_0^2})\,(\frac{t}{t_0})^{2n-2k+1}.
\]
Summing over \(k\le i_n\), the exponential factor is negligible, and the geometric series gives
\[
\sum_{k=1}^{i_n} \frac{\widetilde{R}_{n-k}(t)}{J_{n-k}}
\sim  \frac{\sqrt{\alpha}\,t_0 t}{\sqrt{2\pi n}(t^2-t_0^2)(V'(t)-\frac 2t)}
\exp(-nV(t)+nV(t_0))\,(\frac{t}{t_0})^{2n}.
\]
Therefore, the large deviation for the spectral radius is
\[
\mathbb{P}(\max_{1\le i\le n}|\sigma_i|>t)
\sim \frac{\sqrt{\alpha}\,t_0 t}{\sqrt{2\pi n}(t^2-t_0^2)(V'(t)-\frac 2t)}
\exp(-nV(t)+nV(t_0))\,(\frac{t}{t_0})^{2n}.
\]

Now let \(t=t_0+d_n\) with \(\frac{\sqrt{\log n}}{\sqrt{n}}\ll d_n\ll 1\). A direct calculation gives, uniformly in \(1\le k\le i_n\),
\[
\widetilde{R}_{n-k}(t_0+ d_n)\sim 2\pi e^{-n\varphi_u(t_0+d_n)}\frac{1}{n\varphi_u'(t_0+d_n)}.
\]
From Lemma~\ref{lemphi},
\[
\varphi_u'(t_0+d_n)\sim \alpha  d_n+\frac{2k-1}{n t_0}\sim \alpha d_n
\qquad\text{for } \frac{\sqrt{\log n}}{\sqrt{n}}\ll d_n\ll 1.
\]
Using Lemma~\ref{lemjk}, we obtain
\[
\frac{\widetilde{R}_{n-k}(t_0+d_n)}{J_{n-k}}
\sim \frac{\exp(-n V(t_0+d_n)+nV(t_0))(t_0+d_n)^{2n-2k+1}}
{\sqrt{2\pi \alpha n} \,  d_n\, t_0^{2n-2k+1}}
\exp(-\frac{2k^2}{\alpha n t_0^2}).
\]
Ignoring the exponential factor and summing the geometric series, we get
\[
\sum_{k=1}^{i_n}\frac{\widetilde{R}_{n-k}(t_0+d_n)}{J_{n-k}}
\sim \frac{t_0\exp(-n V(t_0+d_n)+nV(t_0))(t_0+d_n)^{2n}}
{2\sqrt{2\pi \alpha n} \,  d_n^2 t_0^{2n}}.
\]
Combining this with \eqref{connection} completes the proof for this regime. When \(\frac{\sqrt{\log n}}{\sqrt{n}}\ll d_n\ll n^{-1/3}\), using \eqref{modeV} simplifies the above to
\begin{equation}\label{spmoderate}
\mathbb{P}(\max_{1\le i\le n}|\sigma_i|>t_0+ d_n)
\sim \frac{t_0}{2\sqrt{2\pi \alpha n} \, d_n^2 }
\exp(-\frac{1}{2}\alpha n  d_n^2).
\end{equation}

For the particular choice \(d_n=\frac{t\sqrt{\widetilde{\gamma}_n}}{\sqrt{\alpha n}}\) with $t>1$, we still have \(\varphi_u'(t_0+ d_n)\sim \alpha  d_n+\frac{2k-1}{n t_0}\), but the extra term \(\frac{2k-1}{n t_0}\) does not affect the leading order, so \eqref{spmoderate} remains valid. Substituting the explicit form of \(d_n\) yields
\[
\mathbb{P}(\max_{1\le i\le n}|\sigma_i|>t_0+t \frac{\sqrt{\widetilde{\gamma}_n}}{\sqrt{\alpha n}})
\sim t^{-2}(\frac{8\pi (\log n)^2}{\alpha n t_0^2})^{\frac{t^2-1}{2}}.
\]
Finally, if we replace \(t \frac{\sqrt{\widetilde{\gamma}_n}}{\sqrt{\alpha n}}\) by \(\frac{\sqrt{\widetilde{\gamma}_n}}{\sqrt{\alpha n}}+v_n\) (with \(v_n\gg 1\)), the same procedure gives
\[
\mathbb{P}(\max_{1\le i\le n}|\sigma_i|>t_0+\frac{\sqrt{\widetilde{\gamma}_n}}{\sqrt{\alpha n}}+v_n)
\sim \exp(-v_n-\frac{v_n^2}{2\widetilde{\gamma}_n}).
\]
The proof is complete. 

\section{Proof of Theorem \ref{minBerry}}
In this section, we give the proof of Theorem \ref{minBerry}. Recall
\[{\tau}_n=\log \frac{nV''(t_1) t_1^2}{8\pi (\log n)^2}.
\]
We set 
\[
\bar{F}_n(t)=\mathbb{P}(\min_{1\le i\le n }|\sigma_i|\ge t_1-\frac{\sqrt{\tau}_n}{\sqrt{V''(t_1) n}}+\frac{t}{\sqrt{V''(t_1) n\tau_n}}).
\]
Define 
${b}(t)=t_1-\frac{\sqrt{{\tau}_n}}{\sqrt{V''(t_1) n}}(1-\frac{t}{{\tau}_n})$
and the relevant set 
$${B}(t)=\{z: |z|\le {b}(t)\}.$$ 
As for the spectral radius,
$$\bar{F}_n(t)=\det({\rm I}-\mathbb{K}_n|_{{B}(t)})$$ 
and the proof of Theorem \ref{minBerry} relies on the trace and Hilbert-Schmidt norm of $\mathbb{K}_n$ restricted to ${B}(t).$  We state them as a Proposition.
\begin{prop}\label{propleft} Let ${B}(t)$ and $\mathbb{K}_n$ be defined as above. We have 
\begin{equation}\label{Trleft}
{\rm Tr}(\mathbb{K}_n|_{{B}(t)})=e^{t}(1+\frac{2\log\log n}{{\tau}_n}+\frac{2t}{{\tau}_n}-\frac{t^2}{2{\tau}_n})(1+O((\log n)^{-1}))\end{equation}
uniformly on $|t|\lesssim (\log n)^{1/4}.$ Meanwhile, for the Hilbert-Schmidt norm, we have 
\begin{equation}\label{HSleft}
\|\mathbb{K}_n|_{{B}(t)}\|_2\ll e^{t}(\frac{\log n}{n})^{1/4}. \end{equation}
\end{prop}

\begin{proof} 
We first verify the expression of ${\rm Tr}(\mathbb{K}_n|_{{B}(t)}).$
Setting 
$${T}_{k}(t):=2\pi\int_0^{{b}(t)} e^{-nV(r)} r^{2k+1} dr.$$
By definition, 
$${\rm Tr}(\mathbb{K}_n|_{{B}(t)})=\sum_{k=0}^{n-1}\frac{{T}_{k}(t)}{J_k}=\sum_{k=0}^{k_n}\frac{{T}_{k}(t)}{J_k}+\sum_{k_n+1}^{n-1}\frac{{T}_{k}(t)}{J_k},$$ where 
$k_n=\lfloor\sqrt{ t_1^2 V''(t_1) {\tau}_n  n }\rfloor.$ 
Lemma \ref{decreasing} yields 
\begin{equation}\label{decomleft}
	\sum_{k=0}^{k_n}\frac{{T}_{k}(t)}{J_k}\le {\rm Tr}(\mathbb{K}_n|_{{B}(t)})\le \sum_{k=0}^{k_n}\frac{{T}_{k}(t)}{J_k}+ \frac{n{T}_{k_n}(t)}{J_{k_n}}.
\end{equation}
Next, we work on ${T}_{k}(t)$ and $J_k$ for all $0\le k\le k_n.$
Since $r V'(r)$ is strictly increasing, it follows that 
$rV'(r)<0$ for $0<r<t_1.$ Hence $\varphi_u(r)=V(r)-(2-u)\log r$ with $u=2-\frac{2k+1}{n}$ is strictly decreasing on $(0, {b}(t))$ and then the Laplace method gives 
$$ {T}_{k}(t)=2\pi \exp(-n V({b}(t))) ({b}(t))^{2k+2}\frac{1+O(n^{-1})}{-n {b}(t) V'({b}(t))+2k+1 }.$$ 
Analogously, we have
\begin{equation}\label{JKnew} J_{k}=2\pi (1+O(n^{-1}))e^{-nV(t_u)}t_u^{2k+1}\sqrt{\frac{2\pi}{n\varphi''_u(t_u)}}.\end{equation}
Recall $V'(t_u)=(2-u)/t_u.$ For $k\ll n,$ we see 
$t_uV'(t_u)=2-u=o(1)$ and then it follows from the continuity of the function $rV'(r)$ that $t_u=t_1+o(1).$ Thus, the Taylor formula and the fact $V'(t_1)=0$ give 
$$\aligned 
2-u&=t_uV'(t_u)=t_1V''(t_1)(t_u-t_1)+O((t_u-t_1)^2). 
\endaligned $$  
Then $t_u>t_1$ satisfying $t_u-t_1=O(k n^{-1})$ and  
\begin{equation}\label{tuleft} t_u=\frac{2-u}{t_1 V''(t_1) }+t_1+O(k^2n^{-2})=t_1+\frac{2k}{nt_1V''(t_1)}+O(n^{-1}\log n).\end{equation}	
Consequently, we have 
$$\aligned V(t_u)&=V(t_1)+\frac{1}{2}V''(t_1)(t_u-t_1)^2+O(n^{-3/2}(\log n)^{3/2})\\
&=V(t_1)+\frac{2k^2}{n^2 t_1^2V''(t_1)}+O(n^{-3/2}(\log n)^{3/2})
\endaligned $$
and similarly 
$$\varphi_u''(t_u)=V''(t_1)+O(k n^{-1}).$$ 
Putting these two asymptotics into the expression  \eqref{JKnew} of $J_k,$ we obtain 
\begin{equation}\label{Jknewtilde} J_k=2\pi \sqrt{\frac{2\pi}{nV''(t_1)}}(1+O(k n^{-1})) e^{-n V(t_1)} t_u^{2k+1}.\end{equation}
Now 
$${b}(t)V'({b}(t))=-t_1\frac{\sqrt{V''(t_1){\tau}_n}}{\sqrt{ n}}(1-\frac{t}{{\tau}_n})+O(n^{-1}\log n),$$ 
 whence 
$$\aligned 2k+1-n{b}(t)V'({b}(t))&=t_1\sqrt{n V''(t_1){\tau}_n}(1-\frac{t}{{\tau}_n})+2k+O(\log n)\\
&=t_1\sqrt{n V''(t_1){\tau}_n}(1-\frac{t}{{\tau}_n}+\frac{2k}{t_1\sqrt{n V''(t_1){\tau}_n}})(1+O(\frac{\sqrt{\log n}}{\sqrt{n}})).\endaligned $$
Also, 
$$\aligned V({b}(t))&=V(t_1)+\frac{V''(t_1)}{2}({b}(t)-t_1)^2+O(|{b}(t)-t_1|^3)\\
&=V(t_1)+\frac{{\tau}_n}{2n}(1-\frac{t}{{\tau}_n})^2+O(n^{-3/2}(\log n)^{3/2}), \endaligned $$
Thereby, we derive 
\begin{equation}\label{tildeLknew}{T}_k(t)=\frac{2\pi ({b}(t))^{2k+2}(1+O(n^{-1/2}(\log n)^{3/2}))}{t_1\sqrt{n V''(t_1){\tau}_n}(1-\frac{t}{{\tau}_n}+\frac{2k}{t_1\sqrt{n V''(t_1){\tau}_n}})}\exp(-n V(t_1)-\frac{{\tau}_n}{2}(1-\frac{t}{{\tau}_n})^2).\end{equation} 
Combining \eqref{Jknewtilde} and \eqref{tildeLknew}, we have 
$${P}_k(t):=\frac{{T}_k(t)}{J_k}=\frac{({b}(t))^{2k+2}(1+O(n^{-1/2}(\log n)^{3/2})}{t_u^{2k+1}t_1\sqrt{2\pi{\tau}_n}(1-\frac{t}{{\tau}_n}+\frac{2k}{t_1\sqrt{n V''(t_1){\tau}_n}})}\exp(-\frac{{\tau}_n}{2}(1-\frac{t}{{\tau}_n})^2).$$
Using the definition of ${b}(t)$ and \eqref{tuleft}, together with the Taylor formula, we get 
$$\aligned \frac{({b}(t))^{2k+2}}{t_u^{2k+1}t_1}&=(1-\frac{\sqrt{{\tau}_n}}{t_1\sqrt{V''(t_1) n}}(1-\frac{t}{\sqrt{{\tau}}_n})-\frac{2k}{n t_1^2 V''(t_1))})^{2k}(1+O((\log n)^{3/2} n^{-1/2}))\\
&=\exp(-\frac{2k\sqrt{{\tau}_n}}{t_1\sqrt{V''(t_1) n}}(1-\frac{t}{\sqrt{{\tau}}_n})-\frac{4k^2}{n t_1^2 V''(t_1))})(1+O((\log n)^{3/2} n^{-1/2})).
\endaligned $$
It follows that 
\begin{equation}\label{quotienleft}\aligned{P}_k(t)&=\frac{(1+O(n^{-1/2}(\log n)^{3/2})}{\sqrt{2\pi{\tau}_n}(1-\frac{t}{{\tau}_n}+\frac{2k}{t_1\sqrt{n V''(t_1){\tau}_n}})}\\
&\quad \times \exp(-\frac{2k\sqrt{{\tau}_n}}{t_1\sqrt{V''(t_1) n}}(1-\frac{t}{{\tau}_n})-\frac{4k^2}{n t_1^2 V''(t_1))}-\frac{{\tau}_n}{2}(1-\frac{t}{{\tau}_n})^2). \endaligned \end{equation}
Similarly as for the sum $\sum_{j=1}^{j_n}r_k(t),$ the factor 
$$\exp(-\frac{4k^2}{n t_1^2V''(t_1)})(1+\frac{2k}{t_1\sqrt{n V''(t_1){\tau}_n}(1-\frac{t}{{\tau}_n})})^{-1}$$ is negligible while offering the error $O((\log n)^{-1})$ for the sum $\sum_{k=0}^{k_n}{P}_k(t).$ Therefore, it follows from \eqref{quotienleft} that 
 \begin{align}\label{sumquotienleft}\sum_{k=0}^{k_n}{P}_k(t)&=\frac{(1+O((\log n)^{-1})}{\sqrt{2\pi{\tau}_n}(1-\frac{t}{{\tau}_n})}\exp(-\frac{{\tau}_n}{2}(1-\frac{t}{{\tau}_n})^2)\sum_{k=0}^{k_n}\exp(-\frac{2k\sqrt{{\tau}_n}}{t_1\sqrt{V''(t_1) n}}(1-\frac{t}{{\tau}_n}))\notag\\
 &=\frac{(1+O((\log n)^{-1})}{\sqrt{2\pi{\tau}_n}(1-\frac{t}{{\tau}_n})}\exp(-\frac{{\tau}_n}{2}(1-\frac{t}{{\tau}_n})^2)\times\frac{1-\exp(-\frac{2(k_n+1)\sqrt{{\tau}_n}}{t_1\sqrt{V''(t_1) n}}(1-\frac{t}{{\tau}_n}))}{1-\exp(-\frac{2\sqrt{{\tau}_n}}{t_1\sqrt{V''(t_1) n}}(1-\frac{t}{{\tau}_n}))}
 . 
  \end{align}
The choice $k_n=\sqrt{t_1^2V''(t_1)\tau_n n}$ gives $$\exp(-\frac{2(k_n+1)\sqrt{{\tau}_n}(1-\frac t{{\tau}_n})}{t_1\sqrt{V''(t_1) n}})\sim\exp(-2{\tau}_n+2t)\ll (\log n)^{-1}.$$
The Taylor formula yields 
$$1-\exp(-\frac{2\sqrt{{\tau}_n}}{t_1\sqrt{V''(t_1) n}}(1-\frac{t}{{\tau}_n}))=\frac{2\sqrt{{\tau}_n}}{t_1\sqrt{V''(t_1) n}}(1-\frac{t}{{\tau}_n})(1+o((\log n)^{-1})).$$
Therefore, we have from \eqref{sumquotienleft} and these two facts that 
$$\sum_{k=0}^{k_n}{P}_k(t)=\frac{t_1\sqrt{V''(t_1) n}(1+O((\log n)^{-1}))}{2\sqrt{2\pi }{\tau}_n(1-\frac{t}{{\tau}_n})^2}\exp(-\frac{{\tau}_n}{2}(1-\frac{t}{{\tau}_n})^2).$$
Now ${\tau}_n=\log \frac{nV''(t_1) t_1^2}{8\pi (\log n)^2},$ whence 
$$\exp(-\frac{{\tau}_n}{2})=\frac{2\sqrt{2\pi} \log n}{\sqrt{V''(t_1) n} \,t_1}.$$ Therefore, 
\begin{equation}\label{partialsumleft}\sum_{k=0}^{k_n}{P}_k(t)=\frac{e^{t}\log n }{{\tau}_n(1-\frac{t}{{\tau}_n})^2}\exp(-\frac{t^2}{2 {\tau}_n})(1+O((\log n)^{-1})).\end{equation} 
While for the remainder, we get from \eqref{quotienleft} and $k_n=t_1\sqrt{V''(t_1) n{\tau}_n}+O(1)$ that 
\begin{equation}\label{partial0sumleft}\aligned n{P}_{k_n}(t)&\asymp \frac{n}{\sqrt{\log n}}\exp(-\frac{2k_n\sqrt{{\tau}_n}}{t_1\sqrt{V''(t_1) n}}(1-\frac{t}{{\tau}_n})-\frac{4k_n^2}{n t_1^2 V''(t_1))}-\frac{{\tau}_n}{2}+t)\\
&\ll \exp(3 t-5{\tau}_n).\endaligned \end{equation} 
Plugging the expressions \eqref{partialsumleft} and \eqref{partial0sumleft} into \eqref{decomleft} that 
$${\rm  Tr}(\mathbb{K}_n|_{{B}(t)})=\frac{e^{t}\log n }{{\tau}_n(1-\frac{t}{{\tau}_n})^2}\exp(-\frac{t^2}{2 {\tau}_n})(1+O((\log n)^{-1})).$$
Now ${\tau}_n=\log \frac{V''(t_1) n t_1^2}{8\pi (\log n)^2}$ gives 
$$\log n={\tau}_n+2\log\log n+O(1).$$
The choice of $|t|\lesssim (\log n)^{1/4},$ together with corresponding Taylor's formulas, yield in further that 
$${\rm  Tr}(\mathbb{K}_n|_{{B}(t)})=e^{t}(1+\frac{2\log\log n}{{\tau}_n}+\frac{2t}{{\tau}_n}-\frac{t^2}{2{\tau}_n})(1+O((\log n)^{-1})).$$
We now work on the Hilbert-Schmidt  norm. 
Define analogously
\[
{N}_{j, k}(t)=
\frac{1}{\sqrt{J_{j}J_{k}}}
\int_{{B}(t)} e^{-nV(|z|)} z^{k}\bar{z}^{j}\,dz
\]
for $0\le j, k\le n-1.$ 
Writing \(z=re^{i\theta}\), for \(j\neq k\) we have
\[
{N}_{j,k}(t)=
\frac{1}{\sqrt{J_{j}J_{k}}}
\int_0^{{b}(t)}e^{-nV(r)}r^{2n-k-j+1}\,dr
\int_0^{2\pi} e^{i(j-k)\theta}\,d\theta=0
\]
and for $j=k,$  
$${N}_{k, k}(t)=\frac{{T}_k(t)}{J_k}={P}_k(t).$$
Consequently,
\[
\|\mathbb{K}_n|_{{A}(t)}\|_2^2
=
\sum_{k=0}^{n-1} {P}_k^2(t)
\le
\sum_{k=0}^{k_n}  {P}_k^2(t)+n{P}_{k_n}^2(t).
\]
Using \eqref{quotienleft}, we get
\[
{P}_k^2(t)
\lesssim
\frac{1}{\log n}
\times \exp(-\frac{4k\sqrt{{\tau}_n}}{t_1\sqrt{V''(t_1) n}}-{\tau}_n+2t).
\]
Hence, the geometric sum $$\sum_{k=0}^{+\infty}\exp(-\frac{4k\sqrt{{\tau}_n}}{t_1\sqrt{V''(t_1) n}})=\frac{1}{1-\exp(-\frac{4\sqrt{{\tau}_n}}{t_1\sqrt{V''(t_1) n}})}\asymp \frac{\sqrt{n}}{\sqrt{\log n}}$$ yields  
$$\sum_{k=0}^{k_n}{P}_k^2(t)\lesssim \frac{\sqrt{n}}{(\log n)^{3/2}}
\exp(-{\tau}_n+2t)\asymp e^{2t} n^{-\frac12}\sqrt{\log n}.$$
Also, from \eqref{partial0sumleft}, we have
\[
n{P}_{k_n}^2(t)
\asymp \frac{n}{\log n}\exp(-\frac{4k_n\sqrt{{\tau}_n}}{t_1\sqrt{V''(t_1) n}}-\frac{8k_n^2}{n t_1^2 V''(t_1))}-{\tau}_n+2t)
\ll \frac{e^{2t} \sqrt{\log n}}{\sqrt{n}}.
\]
Therefore,
\[
\|\mathbb{K}_n|_{{B}(t)}\|_2 \lesssim \frac{e^{t}(\log n)^{1/4}}{n^{1/4}}.
\]
The proof is completed now. 
\end{proof}

\begin{proof}[{\bf Proof of Theorem \ref{minBerry}}] 
Recall $\ell_1(n)=\frac{1}{4}\log\log n$ and $\ell_2(n)=(\log n)^{1/4}.$ First, \eqref{Trleft} gives 
$$e^{t}-{\rm Tr}(\mathbb{K}_n|_{{B}(t)})=e^t\frac{4\log\log n+4t-t^2+O(1)}{2{\tau}_n},$$ which is $o(1)$ uniformly for $-\ell_2(n)\le t\le \ell_1(n).$ 
Thus 
	$$\aligned |\exp(-{\rm Tr}(\mathbb{K}_n|_{{B}(t)}))-e^{-e^{t}}|&=e^{-e^{t}}|\exp(e^{t}-{\rm Tr}(\mathbb{K}_n|_{{B}(t)}))-1|\\
	&\sim e^{-e^{t}}|e^{t}-{\rm Tr}(\mathbb{K}_n|_{{B}(t)})|\\
	&\sim \frac{1}{2{\tau}_n}e^{-e^{t}+t}|4\log\log n+4t-t^2+O(1)|. 
\endaligned$$
	By \eqref{HSleft} and the inequality \eqref{comp}, we derive similarly as for $\widetilde{X}_n$ that  
$$\sup_{-\ell_2(n)\le t\le \ell_1(n)}|\mathbb{P}(\bar{F}_n(t)-e^{-e^{t}}|\sim\sup_{-\ell_2(n)\le t\le \ell_1(n)}|\exp(-{\rm Tr}(\mathbb{K}_n|_{{B}(t)}))-e^{-e^{t}}|.$$

Thus
\begin{align}\label{middleleft}
\sup_{-\ell_2(n)\le t\le \ell_1(n)}|\mathbb{P}(\bar{F}_n(t)-e^{-e^{t}}|
&\sim \frac{1}{2\tau_n}\sup_{-\ell_2(n)\le t\le \ell_1(n)} e^{-e^{t}+t}|4\log\log n+4t-t^2+O(1)| \notag\\
&\sim \frac{2\log\log n}{e \log n}.
\end{align}
For the final asymptotic equivalence, we rely on the following facts: $\sup_{t\in\mathbb R} e^{-e^{t}+t}|t|^k$ is finite for every fixed $k$, $\sup_{t\in\mathbb R} e^{-e^{t}+t}=1/e$, and ${\tau}_n\sim \log n$.
Now 
\begin{align}\label{leftleft}
	\sup_{t\le -\ell_2(n)}|\bar{F}_n(t)-e^{-e^{t}}|&\le \sup_{t\le -\ell_2(n)}(1-\bar{F}_n(t))+\sup_{t\le -\ell_2(n)}(1-e^{-e^{t}})\notag\\
	&=1-\bar{F}_n(-\ell_2(n))+1-e^{-e^{-\ell_2(n)}}\notag\\
&\asymp e^{-\ell_2(n)}\asymp \exp(-(\log n)^{1/4}).\end{align}
Similarly 
	\begin{align}\label{rightleft}
	\sup_{t\ge \ell_1(n)}|\bar{F}_n(t)-e^{-e^{t}}|&\le \sup_{t\ge \ell_1(n)}\bar{F}_n(t)+\sup_{t\ge \ell_1(n)}e^{-e^{t}}\notag\\
	&=\bar{F}_n(\ell_1(n))+e^{-e^{\ell_1(n)}}\notag\\
&\asymp e^{-e^{\ell_1(n)}}\asymp  \exp(-(\log n)^{1/4}).\end{align} 
Combining \eqref{middleleft}, \eqref{leftleft} and 
\eqref{rightleft} together, we get the desired Berry-Esseen bound 
$$\sup_{t\in\mathbb{R}}|\bar{F}_n(t)-e^{-e^{t}}|\sim\frac{2\log\log n}{e\log n}.$$
The proof for the Berry-Esseen bound is complete.  
Now we work on the large and moderate deviations for $\min_{1\le i\le n}|\sigma_i|.$ 

For \(t<t_1\), define
\[
{L}_{k}(t)=\int_{|z|\le t} \exp(-nV(|z|))\,|z|^{2k}\,d^2z,\qquad 0\le k\le n-1.
\]
Analogously to the spectral radius, we have the basic observation 
$$\aligned
\mathbb{P}(\min_{1\le i\le n}|\sigma_i|<t)=1-\mathbb{P}(\min_{1\le i\le n}|\sigma_i|>t)=1-\det({\rm I}-\mathbb{K}_n|_{\{|z|\le t}\}). \endaligned$$ 
Then, once $\operatorname{Tr}(\mathbb{K}_n|_{\{|z|\le t\}})=o(1),$ we will derive the following crucial asymptotic equivalence  
\begin{equation}\label{connection}
\mathbb{P}(\min_{1\le i\le n}|\sigma_i|<t)\sim \operatorname{Tr}(\mathbb{K}_n|_{\{|z|\le t\}})
=\sum_{k=0}^{n-1} \frac{{L}_{k}(t)}{J_{k}}
\end{equation}
once $\operatorname{Tr}(\mathbb{K}_n|_{\{|z|\le t\}})=o(1).$
The same saddle-point analysis yields as for $\widetilde{T}_k(t)$, uniformly for \(0\le k\le k_n\), it holds 
\begin{equation}\label{tildeLleft}
{L}_{k}(t)\sim 2\pi \exp(-n V(t)) t^{2k+2}\frac{1}{-n t V'(t)+2k+1 }.
\end{equation}
Thereby, it follows from \eqref{tuleft} and \eqref{Jknewtilde} that 
\[\aligned 
\frac{{L}_{k}(t)}{J_{k}}
&\sim \frac{\sqrt{V''(t_1)}}{\sqrt{2\pi n}(-V'(t))}
\exp(-nV(t)+nV(t_1))\,(\frac{t}{t_1})^{2k+1}(1+\frac{2k}{nt_1^2V''(t_1)})^{-2k-1}\\
&\sim \frac{\sqrt{V''(t_1)}}{\sqrt{2\pi n}(-V'(t))}
\exp(-nV(t)+nV(t_1))\,(\frac{t}{t_1})^{2k+1} \exp(-\frac{4k^2}{n t_1^2 V''(t_1)}). \endaligned 
\]
Summing over \(0\le k\le k_n\), the exponential factor is negligible, and the geometric series gives
\[
\sum_{k=0}^{k_n} \frac{{L}_{k}(t)}{J_{k}}
\sim  \frac{\sqrt{V''(t_1)}\,t_1 t}{\sqrt{2\pi n}(t_1^2-t^2)(-V'(t))}
\exp(-nV(t)+nV(t_1))
\]
as well as 
$$\sum_{k=k_n+1}^{n-1}\frac{{L}_{k}(t)}{J_{k}} \lesssim \frac{n{L}_{k_n}(t)}{J_{k_n}}\sim \sqrt{n}
\exp(-nV(t)+nV(t_1))\,(\frac{t}{t_1})^{2k_n+1} \exp(-\frac{4k_n^2}{n t_1^2 V''(t_1)}).$$ 
Now
$$\frac{\sum_{k=k_n+1}^{n-1}\frac{{L}_{k}(t)}{J_{k}}}{\sum_{k=0}^{k_n} \frac{{L}_{k}(t)}{J_{k}}}\sim n (\frac{t}{t_1})^{2k_n+1} \exp(-\frac{4k_n^2}{n t_1^2 V''(t_1)})\ll 1,$$ which implies 
$${\rm Tr}(\mathbb{K}_n|_{\{|z|\le t\}})\sim \sum_{k=0}^{k_n}\frac{{L}_{k}(t)}{J_{k}}\asymp \frac{1}{\sqrt{n}}\exp(-nV(t)+nV(t_1))\ll 1.$$   Then 
$$\mathbb{P}(\min_{1\le i\le n}|\sigma_i|<t)\sim \sum_{k=0}^{k_n} \frac{{L}_{k}(t)}{J_{k}}\sim \frac{\sqrt{V''(t_1)}\,t_1 t}{\sqrt{2\pi n}(t_1^2-t^2)(-V'(t))}
\exp(-nV(t)+nV(t_1)).$$
The large deviation is established.

Now let \(t=t_1-d_n\) with \(\frac{\sqrt{\log n}}{\sqrt{n}}\ll d_n\ll 1\). Inserting $t=t_1-d_n$ into \eqref{dnleft}, we have
\[
{L}_{k}(t_1-d_n)\sim \frac{2\pi \exp(-n V(t_1-d_n)) (t_1-d_n)^{2k+2}}{-n (t_1-d_n) V'(t_1-d_n)+2k+1 }.
\]
 
We use the Taylor expansion, the assumption $V'(t_1)=0$ and the restraint $\sqrt{\log n}/\sqrt{n}\ll d_n\ll 1$ to get 
$$-n (t_1-d_n) V'(t_1-d_n)+2k+1=n d_n t_1V''(t_1)+2k+1+O(nd_n^2)\sim t_1V''(t_1)n d_n $$
uniformly in \(0\le k\le k_n\). Thus, 
$$\frac{{L}_k(t_1-d_n)}{J_k}\sim \frac{(1-\frac{d_n}{t_1})^{2k+1}}{\sqrt{2\pi V''(t_1) n d_n} }\exp(-nV(t_1-d_n)+n V(t_1)-\frac{4k^2}{t_1^2V''(t_1) n}).$$
Similarly as above, we have 
\begin{align}\label{dnleft} \mathbb{P}(\min_{1\le i\le n}|\sigma_i|\le t_1-d_n)&\sim \sum_{k=0}^{n-1}\frac{(1-\frac{d_n}{t_1})^{2k+1}}{\sqrt{2\pi V''(t_1) n }d_n }\exp(-nV(t_1-d_n)+n V(t_1)) \notag\\
&\sim \frac{t_1}{2\sqrt{2\pi V''(t_1) n }d_n ^2} \exp(-nV(t_1-d_n)+n V(t_1)).\end{align}

When \(\frac{\sqrt{\log n}}{\sqrt{n}}\ll d_n\ll n^{-1/3}\), the asymptotic \eqref{dnleft}  is improved to 
\begin{equation}\label{moderateleft}
\mathbb{P}(\min_{1\le i\le n}|\sigma_i|\le t_1-d_n)
\sim \frac{t_1}{2\sqrt{2\pi V''(t_1) n} \, d_n^2 }
\exp(-\frac{1}{2}V''(t_1) n  d_n^2).
\end{equation}

For the particular choice \(d_n=\frac{t\sqrt{{\tau}_n}}{\sqrt{V''(t_1) n}}\) with $t>1$, we will have $$-n (t_1-d_n) V'(t_1-d_n)+2k+1\sim n d_n t_1V''(t_1)(1+\frac{2k}{t_1V''(t_1) n d_n}).$$ The extra term \(\frac{2k}{t_1 V''(t_1)n d_n}\) does not contribute to the leading order as for the spectral radius, so \eqref{moderateleft} remains valid. Substituting the explicit form of \(d_n\) yields
\[
\mathbb{P}(\min_{1\le i\le n}|\sigma_i|<t_1- \frac{t\sqrt{{\tau}_n}}{\sqrt{V''(t_1) n}})
\sim t^{-2}(\frac{8\pi (\log n)^2}{V''(t_1) n t_1^2})^{\frac{t^2-1}{2}}.
\]
Finally, if we replace \( \frac{t\sqrt{{\tau}_n}}{\sqrt{V''(t_1) n}}\) by \(\frac{\sqrt{{\tau}_n}}{\sqrt{V''(t_1) n}}(1+\frac{v_n}{\tau_n})\) (with \(1\ll v_n\ll \tau_n\)) in \eqref{moderateleft}, simple calculus gives
\[
\mathbb{P}(\min_{1\le i\le n}|\sigma_i|<t_1-\frac{\sqrt{{\tau}_n}}{\sqrt{V''(t_1) n}}+v_n)
\sim \exp(-v_n-\frac{v_n^2}{2{\tau}_n}).
\]
The proof is complete.

\end{proof}

\end{document}